\documentclass[reqno]{amsart} 
\usepackage{graphicx} 

\numberwithin{equation}{section} 

\usepackage{enumitem}

\usepackage{amssymb,amsmath,latexsym,mathtools,tensor}

\usepackage[font=small,labelfont=bf]{caption} 

\usepackage{etoolbox}

\usepackage{booktabs}

\usepackage{hyperref}

\expandafter\def\expandafter\normalsize\expandafter{%
    \normalsize
    \setlength\abovedisplayskip{10pt}
    \setlength\belowdisplayskip{10pt}
    \setlength\abovedisplayshortskip{10pt}
    \setlength\belowdisplayshortskip{10pt}
}

\makeatletter

\usepackage{epigraph}

\patchcmd{\epigraph}{\@epitext{#1}}{\itshape\@epitext{#1}}{}{}
\makeatother

\everymath=\expandafter{\the\everymath\displaystyle}

\usepackage{tikz-cd}

\usepackage{tikz}
\usetikzlibrary{fit}

\usepackage{blkarray}

\usetikzlibrary{patterns}
\usetikzlibrary{calc}

\pgfdeclarelayer{edgelayer}
\pgfdeclarelayer{nodelayer}
\pgfsetlayers{edgelayer,nodelayer,main}

\newtheorem{theorem}{Theorem}[section]
\newtheorem{lemma}[theorem]{Lemma}
\newtheorem{proposition}[theorem]{Proposition}
\newtheorem{corollary}[theorem]{Corollary}

\newcommand{\Mod}[1]{\ (\mathrm{mod}\ #1)} 

\newcommand{\image}{{\rm{im} }\,}

\newcommand{\kernel}{{\rm{ker} }\,}

\newcommand{\rk}{{\rm{rk} }\,}

\newcommand{\real}{{\Bbb R}}
\newcommand{\comp}{{\Bbb C}}

\newcommand{\irrep}[1]{{\mathcal V}_{#1}}

\newcommand{\calv}{{\mathcal V}}

\newcommand{\cbra}[3]{\left\langle #1,#2 \right\rangle_{#3}}
\newcommand{\sqb}[2]{\left[ #1,#2 \right]}

\newcommand{\sltc}{{SL_2(\comp)}}
\newcommand{\liesltc}{{{\frak sl}_2(\comp)}}

\newcommand{\gz}{{\frak g_0}}
\newcommand{\gp}{{\frak g_+}}
\newcommand{\gm}{{\frak g_-}}
\newcommand{\gn}[1]{{\frak g_{#1}}}
\newcommand{\hlie}{{\frak h}}
\newcommand{\glie}{{\frak g}}
\newcommand{\alie}{{\frak a}}
\newcommand{\klie}{{\frak k}}

\newcommand{\bab}{{\langle a, b \rangle}}
\newcommand{\buat}{{\langle u, a^3 \rangle}}
\newcommand{\buadb}{{\langle u, a^2 \,b \rangle}}
\newcommand{\buabd}{{\langle u, a \, b^2 \rangle}}
\newcommand{\bubt}{{\langle u, b^3 \rangle}}

\newcommand{\jimport}[1]{ graphics}  

 \numberwithin{repeat}{repsection}

\begin{document}

\title[Okamoto Transformations and Instantons]{ 
Okamoto Transformations relating Equivariant Instanton Bundles via Painlev\'e VI}

\author{Jan Segert}
\address{Mathematics Department\\
         University of Missouri\\
         202 Math Sci Bldg, MO 65211}
\email[J.~Segert]{segertj@umsystem.edu}

\date{March 10, 2021}

\maketitle

\begin{abstract}
We compute explicit solutions $\Lambda^\pm_m$ of the Painlev\'e VI (PVI) differential equation from  
 equivariant instanton bundles $E_m$ corresponding to Yang-Mills instantons with ``quadrupole symmetry."  
 This is based on a  generalization of Hitchin's logarithmic connection to vector bundles with an $\sltc$ action.   
We then identify  explicit Okamoto transformation which play the role of ``creation operators"  for 
construction $\Lambda^\pm_m$ from the ``ground state" $\Lambda^\pm_0$, suggesting that the equivariant 
instanton bundles $E_m$ might similarly be related to the trivial ``ground state" $E_0$.  
\end{abstract}

\tableofcontents


\section{Introduction}
\subsection{The Hierarchy of Equivariant Instanton Bundles}

The work of Atiyah-Drinfeld-Hitchin-Manin (ADHM) on   instantons \cite{ADHM, At1} is an important 
milestone for the highly productive interplay between mathematics and physics, and continues to  
inspire interesting work to this day.   
ADHM studied the correspondence between the  {\it Yang-Mills instantons} of theoretical physics 
and  {\it  instanton bundles}.  Yang-Mills instantons are
 self-dual  connection on a vector bundle $F$ over the four-sphere $S^4$.   
 The  {\it Penrose twistor space} of $S^4$ is the complex projective space $\comp P^3$. 
 Instanton bundles are a class of holomorphic vector bundles $E$ over $P^3$   
 constructed from the linear algebra data of a   {\it self-dual linear monad}, as discussed in Section \ref{sec:adhm} and Appendix \ref{sec:monads}.

 The starting point for the present paper is a result of Gil Bor and the author \cite{BS} on {\it equivariant} Yang-Mills instantons. 
 Specifically, we studied self-dual connections with a certain {\it quadrupole symmetry}.
 This work resolved  open questions from  \cite{SS,Bor2},   where  quadrupole symmetry 
 played a key role in establishing the existence of {\it non-minimal} (neither self-dual nor anti-self-dual) Yang-Mills connections  with nonzero 
 instanton number $c_2(E) \ne 0$.   The existence of 
 non-minimal Yang-Mills connections had been established for the $c_2(E) = 0$ case in the seminal work \cite{SSU}.  
Quadrupole symmetry of instantons on $S^4$ corresponds to the  action of $SU(2)$ on the the complex twistor space
 space $P^3 = P(\comp^4)$ via the unique  irreducible  representation on $\comp^4$.   The Fano three-fold 
 $P^3$ is then almost-homogenous under the complexification of the $SU(2)$ action to an $\sltc$ symmetry, meaning that the $\sltc$ action has  an open dense orbit.  The key existence \cite{BS} and uniqueness \cite{SegU} results from Section \ref{sec:adhm} are summarized in:

\begin{theorem} 
\label{thm:BS}   
For each nonnegative integer $m$ there exists an $\sltc$ equivariant  instanton bundle $E_m$ with rank $\rk E_m = 2$ and instanton number \nobreak{
$c_2(E_m) = \tfrac{1}{2} m (m+1)$}.      Every $\sltc$ equivariant instanton bundle $E$ of rank $\rk E = 2$ is isomorphic 
to one of the $E_m$.  
\end{theorem}

Allowing for a modicum of poetic license, the hierarchical structure of Theorem \ref{thm:BS}  is reminiscent of  
the familiar quantum harmonic oscillator \cite{qho}, where the energy eigenstates  $\psi_m$ are 
similarly indexed by a nonnegative integer $m$.    Furthermore, any eigenstate $\psi_m$ is constructed 
from the ground state $\psi_0$ by successive applications of Dirac's famous {\it creation operator} ${\mathbf a}^\dagger$. 
A motivation for this paper was to explore the possibility of an analogous ``creation 
operator" $S$ which by successive applications would construct the equivariant instanton bundle $E_m$ from the 
trivial ``ground state" equivariant instanton bundle $E_0$.  
This putative ``creation operator" $S$ is indicated by dashed arrows in the schematic:  
\begin{align}
\label{eq:creation}
&
\begin{tikzcd}[row sep=normal, column sep=large,ampersand replacement=\&]
E_0  \arrow[r,dashed,"S"] 
\& E_1  \arrow[r,dashed,"S"] 
\&  E_2  \arrow[r,dashed,"S"] 
\& E_3  \arrow[r,dashed,"S"] 
\& E_4. 
\end{tikzcd}
\end{align}
\noindent  Alas we did not manage to construct such a   ``creation operator" $S$ for the hierarchy $E_m$ of equivariant instanton bundles 
in Eqn.\,\ref{eq:creation}.  

\vfill\eject

\subsection{The Shadow Hierarchy of PVI Solutions} 
We do however encounter  some interesting ``Platonic shadows" \cite{plato} of the putative ``creation operator" $S$, as summarized in the 
follow schematic subsuming eqn.\,\ref{eq:creation}: 
\begin{align}
\label{eq:combined} 
\begin{tikzcd}[row sep=normal, column sep=large,ampersand replacement=\&]
 \Lambda_0^+ \arrow[r,"Q"]   
 \&  
  \Lambda_1^+ \arrow[r,"Q"]  
\& 
 \Lambda_2^+ \arrow[r,"Q"]  
 \& 
  \Lambda_3^+ \arrow[r,"Q"]  
 \&  
 \Lambda_4^+  
\\
E_0 \arrow[u] \arrow[d] \arrow[r,dashed,"S"]   
\& E_1 \arrow[u] \arrow[d] \arrow[r,dashed,"S"]   
\& E_2 \arrow[u] \arrow[d] \arrow[r,dashed,"S"]   
\& E_3\arrow[u] \arrow[d] \arrow[r,dashed,"S"]   
\&E_4 \arrow[u] \arrow[d]
\\
 \Lambda_0^- \arrow[r,"Q^{-1}"]  
 \&   
 \Lambda_1^- \arrow[r,"Q^{-1}"] 
 \& 
 \Lambda_2^- \arrow[r,"Q^{-1}"] 
 \& 
 \Lambda_3^- \arrow[r,"Q^{-1}"] 
\& \Lambda_4^- 
\end{tikzcd}
\end{align}
\noindent  As further discussed in Sections \ref{sec:biglog} and \ref{sec:pi},    
each of the equivariant instanton bundles $E_m$ of 
Theorem  \ref{thm:BS} yields a pair of PVI solutions:  $\Lambda^+_m \leftarrow E_m \rightarrow \Lambda^-_m$. 
We  have explicitly computed $\Lambda^\pm_m$ for $0 \le m \le 4$ in Section \ref{sec:painleve}, following \cite{Seg1} and \cite{B}, see also \cite{Man1}.   
Other geometric constructions of algebraic Painlev\'e VI solutions, and  associated  Frobenius manifolds, include  \cite{H1,D,DM,H2,H3,H4,Seg2}.
This computation generalizes  the work Hitchin in the influential paper \cite{H1}, which gives  
the $m = 0$ ``ground state" case  $\Lambda^+_0 \leftarrow E_0 \rightarrow \Lambda^-_0$.  
Transformations between solutions of  Painlev\'e VI have been known for many years, but deeper insights into their structure 
continue to emerge \cite{DR, Bo3,Bo1,Fil}.    In honor of Okamoto's decisive contribution to this area \cite{Oka}, we will 
use the term  {\it Okamoto transformation} for all  transformations between  Painlev\'e VI solutions.   Section \ref{sec:okamoto} contains additional discussion and references.  
As further  discussed in Section \ref{sec:okamoto}, 
each solid horizontal arrow $Q: \Lambda^+_m \to \Lambda^+_{m+1}$ and 
$Q^{-1}: \Lambda^-_{m} \to \Lambda^-_{m+1}$ 
indicates  an explicit Okamoto transformation. 

Then the operators $Q$ and $Q^{-1}$ each act as a ``creation operator" on its respective hierarchy of PVI solutions,  
conjecturally for all nonnegative integers $m$ although our case-by-case proof only covers a finite subset:  
\begin{theorem} 
\label{thm:pair}  
For each nonnegative integer $m \le 4$, 
\begin{equation*}
\Lambda_m^+ = Q^{ m} \Lambda_0^+, \qquad  \Lambda_m^- = Q^{- m} \Lambda_0^-.
\end{equation*}
\end{theorem}

\noindent We interpret the 
 ``creation operators" $Q$ and $Q^{-1}$ as ``shadows" of a putative creation operator $S$ for equivariant 
 instanton bundles.   Reconstructing the equivariant instanton $E_m$ from the corresponding pair of Painlev\'e solutions 
$\Lambda^+_m$ and $\Lambda^-_m$ is one possible approach to constructing $S$.  
It is theoretically possible to reconstruct an instanton bundle on $P^3$ from its divisor of jumping lines 
and associated data \cite{OS, Hu,GH}.   
 In our equivariant setting, jumping lines are manifested as poles of the PVI solutions.  
 However,   the  PVI solutions $\Lambda^\pm_m$ have additional poles not corresponding 
 to jumping lines, which have so far thwarted attempts to reconstruct the divisor of jumping lines.

 \subsection{Notes}

One of the motivations for the present paper is a body of  interesting work on the aetiology of the Okamoto transformations from the viewpoint of isomondromic deformations focusing on Katz's {\it middle convolution} and 
related ideas \cite{DR, Bo1,Bo3,Fil,Simpson}, complementing  results on the classical Schlesinger transformations e.g \cite{MS}.  
The hope is that a variant of middle convolution might produce the putative ``creation operator" $S: E_m \to E_{m+1}$ on 
the hierarchy of equivariant instanton bundles.  

Another motivation for the present paper is a wave of recent interest in   
instanton bundles and monads on  Fano threefolds other than the classic case of $P^3$, see 
e.g. \cite{Costa,Kuz,Faenzi,Sanna,Sanna2,Ma1,Ma2,Ma3,Ma4,Ma5,Qin}.  
Like $P^3$, several of these Fano threefolds  are almost-homogeneous for a natural $\sltc$ action
(see  \cite{LV,MJ,Tim,Bri} for  almost-homogenous threefolds).  In one particularly 
promising instance, 
\cite{Sanna} established existence and uniqueness of an equivariant instanton bundle with minimal nonzero instanton numbers. 
It is intriguing to speculate whether the equivariant instanton bundles might also constitute a hierarchy in this case, by analogy to   Theorem \ref{thm:BS} in the case of $P^3$.   Fano threefolds have a plethora of rational curves since they are rationally connected 
\cite{camp,kol,harris}.  This could open the door for some of the ideas of Okamoto transformations, and possibly middle convolution, 
that play an important  role for the $P^3$ results of the present paper.

\section{Representation Theory Prerequisites}
\label{sec:reps}

In this section we start with a unified review of the finite-dimensional 
representation theory of $\sltc$, with  
 emphasis is on the 
decomposition of tensor product representations into 
irreducibles.  Although these results are rather standard,  
the formalism used in the present paper, and in \cite{BS},  is rather nonstandard. 

Let  
$\calv := \comp[x,y]$ be the infinite-dimensional 
vector space of polynomials 
with generated by the two 
indeterminates $x$ and $y$.  
$\sltc$  acts on $\calv$ in the usual way by 
linear substitutions in $x$ and $y$.  The subspace 
$\irrep d$ of degree-$d$ homogeneous polynomials is a 
vector space of dimension $d+1$.  It is well known 
that $\sltc$ acts irreducibly on each $\irrep d$, and that irreducible finite-dimensional representation of $\sltc$  
is isomorphic to some $\irrep d$.

Tensor products of the irreducible representations $\irrep d$ decompose
according to the {\it Clebsch-Gordan formula}
\begin{align}
\irrep i \otimes \irrep j 
= \irrep {i+j}\oplus \irrep {i+j-2}\oplus 
\dots \oplus \irrep {|i-j|}. \label{eq:4.1}
\end{align}
The  
equivariant linear projections implicit in 
the Clebsch-Gordan formula are simply described in terms of 
the ``transvectants'' (\"Uberschiebungen) 
of classical invariant theory.  
For a non-negative integer $p$, the 
{\it $p$-th transvectant} is the  $\sltc$-equivariant    
bilinear map  
$\cbra \cdot \cdot p :  \calv \times \calv   \to \calv $  
defined  by  
\begin{align}
\cbra u v p 
&:= {1 \over p!} \sum_{k=0}^p (-1)^k {p \choose k}
{\partial^p u \over \partial x^{p-k} \partial y^k}
{\partial^p v \over \partial 
x^k \partial y^{p-k}}. \label{eq:4.2}
\end{align}
Verifying the equivariance of the transvectants is not difficult, 
see e.g. \cite{BS}.  
The $0$-th transvectant is ordinary multiplication of polynomials, 
$\cbra u v 0 = u\, v$. 
To relate transvectants to  the Clebsch-Gordan formula, one checks 
that the 
that a restriction of the 
$p$-th transvectant is a bilinear map  
$\irrep i \times \irrep j \to \irrep {i + j - 2p}$ 
which is  nonzero if and only if    $0 \le p \le \min(i,j)$.  

From the  obvious symmetry property 
\begin{align} \cbra u v p = (-1)^p \cbra v u p , \label{eq:4.3} 
\end{align}
we see that the nondegenerate bilinear form $\cbra \cdot \cdot { }: = \cbra \cdot \cdot p:\irrep p \times \irrep p \to \irrep 0$ is 
symmetric if $p$ is even, and antisymmetric (symplectic) if $p$ is odd.  
We denote also by  $\cbra \cdot \cdot { }$ the bilinear extension to any 
direct sum of irreducible representation $\irrep d$.

A basis for  $\irrep p$ may be constructed 
from a pair of linearly independent vectors $a, b \in \irrep 1$.   
Linear independence of $a$ and $b$ is equivalent to $\cbra a b {}  \ne 0$. 
One may  check that 
$\{ a^p, a^{p-1}\, b, \dots,  a \, b^{p-1}, b^p \}$ is a basis 
of $\irrep p$, since 
\begin{align}
\cbra {a^{p-j} \, b^{j}} {a^{j} \, b^{p-j}  } {} =  
(-1)^{j} \, j! \, (p-j)!\,  \cbra a b {}^{p} \ne 0, \label{eq:4.4}
\end{align}
and $\cbra {a^{p-j}\, b^j} {a^k \, b^{p-k}} {} = 0$ if $k \ne j$.   
This simultaneously proves the nondegeneracy of the bilinear form 
$\cbra \cdot \cdot {}$ on $\calv$.

The adjoint representation of $\sltc$ on its Lie algebra 
$\liesltc$ is an irreducible 
representation of dimension three, so it is
isomorphic to the representation $\irrep 2$.    
The Lie bracket is an equivariant antisymmetric bilinear map 
$\sqb \cdot \cdot : \liesltc \times \liesltc \to \liesltc$.  
The first transvectant is an equivariant antisymmetric bilinear map
$\cbra \cdot \cdot 1 : \irrep 2 \times \irrep 2 \to \irrep 2$, 
and modulo scalar multiplication, is the only such map. 
We can then fix an identification of the 
Lie algebras  $\liesltc$ and $\irrep 2$ by defining the 
Lie bracket on $\irrep 2$ to be the first transvectant, i.e.  
$$ \sqb u v := \cbra u v 1, \qquad u,v \in \irrep 2 .  \nonumber $$
A convenient basis of the Lie algebra $\irrep 2$ is 
\begin{align}
 \gz(a,b) := - {a \, b \over \cbra a b {} }, \qquad
 \gp(a,b) := {a^2 \over 2 \, \cbra a b {} }, \qquad
 \gm(a,b) := - {b^2 \over 2 \, \cbra a b {}}, \label{eq:4.5} 
 \end{align}
which will be abbreviated as $\{ \gz, \gp, \gm \}$ when the choice 
of $a$ and $b$ is clear.  The Lie brackets of these basis vectors are: 
\begin{align*}
 \sqb {\gz } {\gp } = 2 \gp , \qquad
\sqb {\gz } {\gm } = - 2 \gm , \qquad
\sqb {\gp } {\gm } = \gz  .  
\end{align*}

More generally, 
the linearization of the $\sltc$ action on $\irrep p$ is an 
equivariant bilinear map  
$\sqb \cdot \cdot: \liesltc \times \irrep p \to \irrep p$.  
Identifying $\liesltc$ with $\irrep 2$ as above, one checks that 
this map again coincides with the first transvectant, 
\begin{align*}
 \sqb \cdot \cdot : \irrep 2 \times \irrep p &\to  \irrep p 
\\
(u,v) &\mapsto \cbra u v 1 .  
\end{align*}
It is straightforward to compute the action of the Lie algebra 
on the monomials in $\calv$,  
\begin{align}
\sqb {\gz(a,b) } {a^i\,  b^j} &= (i-j) \, a^i \, b^j \nonumber \\
\sqb {\gp(a,b) } {a^i\,  b^j} &= j \, a^{i+1} \, b^{j-1} \nonumber \\
\sqb {\gm(a,b) } {a^i\, b^j} &=  i \, a^{i-1} \, b^{j +1} .   \label{eq:4.8}
\end{align}

\section{The Equivariant Instanton Bundles $E_m$}

\subsection{Existence and Uniqueness}
\label{sec:adhm}

In this section we revisit the results of \cite{BS} on equivariant instanton bundles on the complex projective space $P^3$.   
The standard references for instanton bundles and self-dual ADHM 
include \cite{At1,ADHM,OS,OSS}.   Appendix \ref{sec:monads} contains more details for the equivariant case of interest. 
We write $P^3 = P(Z)$ as the projectivization of the complex vector space $Z := \comp^4$.   
Let $W$ be a complex vector space, let  $W'$ be its dual, and let $V$ be a complex vector space endowed with a complex symplectic form 
$\langle \cdot, \cdot \rangle$.  Then a {\it  (self-dual  linear) monad} over  $P^3$ is a   
a complex of vector bundles  
\begin{equation}
\begin{tikzcd}[row sep=normal, column sep=normal]
0 \arrow[r] & W \otimes {\mathcal O}_{P^3}(-1)  \arrow[r,"A"] &  V \otimes  {\mathcal O}_{P^3}  \arrow[r,"A^o"]& W' \otimes  {\mathcal O}_{P^3}(1)  
\arrow[r] & 0, 
\end{tikzcd}
\label{eq:new}
\end{equation}
where $A$ and $A^o$ are  as follows.    
The fundamental datum for the monad is an injective linear map  {$A: W \otimes Z \to V$}.  
This defines the vector bundle map $W \otimes {\mathcal O}_{P^3}(-1) \to V \otimes  {\mathcal O}_{P^3}$. 
Composing the transpose $A': V' \to  W' \otimes Z'$ with the 
canonical isomorphism $s: V \to V'$, defined by $s(v) = \langle v , \cdot \rangle$, yields the surjective 
linear map $A^o :=A' \circ s: V \to  W' \otimes Z'$.  This defines 
the  vector bundle map $V \otimes  {\mathcal O}_{P^3}  \to W' \otimes  {\mathcal O}_{P^3}(1)$.  
These vector bundle maps are required to fit together as a complex,  meaning that the kernel of an outoging 
map contains the image of the incoming map.  This imposes conditions on the 
fundamental datum $A: W\otimes Z \to V$, which are discussed in Proposition \ref{thm:addef}.     

Self-dual linear monads correspond to a class of holomorphic 
vector bundles over $P^3$, which we will call  {\it instanton bundles} in this paper (the term 
 {\it mathematical instanton bundles} is also used).  
The construction an instanton bundle as the cohomology of the monad is straightforward.  
The kernel $U^o := \kernel A^o$ is a  subbundle of the trivial bundle $V \otimes  {\mathcal O}_{P^3}$, with  $\rk U^o = \dim V - \dim W$. 
Similarly, the image $U := \image A$ is a subbundle of $V \otimes  {\mathcal O}_{P^3}$, with rank  $\rk U = \dim W$.  
Furthermore since this is a complex, $U$ is a subbundle of $U^o$  and the cohomology of the complex is the 
instanton bundle  $E : = U^o/U$  over $P^3$.    
 The rank  of the instanton bundle $\rk E = \dim V - 2 \dim W$  is even, since  
 the  symplectic vector space  $V$ has even dimension.  Furthermore the holomorphic 
 vector bundle $E$ inherits a complex symplectic structure, so the the line bundle $\det E$ is trivial 
 and $c_1(E) = c_1(\det E) = 0$.   We are primarily interested in the instanton bundles 
 of rank $\rk E = 2$,  where the  {\it instanton number} (or {\it  charge}) is \cite{At1,ADHM} the integer    $c_2(E) = \dim W \ge0$.

The inverse correspondence, constructing a monad from an instanton bundle,  
is important for studying instanton bundles with a group action.  Representation theory 
the affords  powerful tools for studying   linear maps $A: W \otimes Z \to V$ in the equivariant case. 
The starting point for the present paper is the following result of \cite{BS}.  Here the group $\sltc$ acts on $P^3 = P(Z)$ by 
identifying $Z := \irrep 3$ with the irreducible representation of on $\comp^4$:   

\begin{theorem}  (\cite{BS}) 
\label{thm:bsres}
For each nonnegative integer $m$ there exists an $\sltc$ equivariant  instanton bundle $E_m$ with rank $\rk E_m = 2$ and instanton number $c_2(E_m) = \tfrac{1}{2} m (m+1)$.    
\end{theorem} 

\noindent  The explicit construction from \cite{BS} of the equivariant monads for these $E_m$ is summarized in  Appendix \ref{sec:monads}. 
The paper \cite{BS}  focused on  $SU(2)$ equivariant ``real" instanton bundles, where the monad is required to satisfy  
an additional ``ADHM reality condition."    This is an equivariant case of 
the famous  Atiyah-Drinfeld-Hitchin-Manin correspondence  \cite{At1,ADHM} between ``real" instanton bundles on $P^3$ and 
 Yang-Mills instantons on $S^4$.    The existence result 
Theorem \ref{thm:bsres} follows immediately from  \cite{BS} by forgetting the real structures and complexifying the group action.   

Uniqueness of equivariant instanton bundles depends on the notion of equivalence.      
In \cite{BS} it was shown that every $SU(2)$ equivariant {\it real} instanton bundle of rank $2$ is equivalent to one of the bundles $E_m$ of  Theorem \ref{thm:bsres}.   But this uniqueness result does not preclude the possibility that there could exist other $\sltc$ equivariant rank $2$ instanton bundles that do not admit a  real structure.   That possibility was eliminated in subsequent work: 
 
 \begin{theorem}  (\cite{SegU}) 
\label{thm:nonbs} Every $\sltc$ equivariant instanton bundle $E$ of rank $\rk E = 2$ is isomorphic to one of the instanton bundles 
$E_m$ of Theorem \ref{thm:bsres}. 
\end{theorem} 

\noindent   A corollary is that every $\sltc$ equivariant instanton bundle $E$ of rank $\rk E =2$ admits a real structure.  However that  
this is  true {\it only} for  $\rk E =2$, and fails for higher rank.    
The proof of Theorem \ref{thm:nonbs} in \cite{SegU} reworks the ideas of \cite{BS} within the context  
 complex geometry of the Fano threefold $P^3$, bypassing any steps that require the reality condition or 
 study of Yang-Mills instantons on $S^4$.     
The main ingredient is the Drinfeld-Manin   identification \cite{ADHM, At1} of $A^o: V \to W' \otimes Z'$ 
with a  natural map 
$H^1(E \otimes \Omega^1) \to H^1(E(-1)) \otimes H^0( \mathcal{O}(1))$ of sheaf cohomology groups,  where 
$Z' = H^0( \mathcal{O}(1))$, 
$V =H^1(E \otimes \Omega^1)$,  and $W' =  H^1(E(-1))$ is the dual of $W = H^1(E \otimes \Omega^2(1))$.  
The analogous computation of \cite{BS} for the ``real" case was based on the identification 
of $W$ and $V$ with the kernels of certain Dirac operators \cite{Hitr,At1} on $S^4$ coupled 
to the corresponding Yang-Mills instanton connection.  In both cases the key step is using the  Atiyah-Bott-Lefschetz  fixed-point theorem \cite{AB}  to determine  characters of the admissible group actions on $W$ and $V$.

 \subsection{Construction Overview} 
 \label{sec:monads}

This section contains an overview of instanton bundles from \cite{At1, ADHM}, and 
an overview of the equivariant instanton bundles $E_m$ from  \cite{BS}.  
  The following fundamental result of 
 Atiyah-Drinfeld-Hitchin-Manin \cite{ADHM, At1} translates the self-dual monads of eqn.\,\ref{eq:new} to linear algebra data:

 \begin{proposition}
  \label{thm:addef}   
 Let $W$ be a complex vector space and let $V$ be a complex vector space with symplectic form $\langle \cdot, \cdot \rangle$. 
 A linear map $A: W \otimes Z \to V$ defines a self-dual monad as in eqn.\,\ref{eq:new}  
 iff the following conditions hold:   
 \begin{enumerate}
 \item The ``injectivity condition": $A(w\otimes z) = 0$ only if $w \otimes z = 0$. 
 \item The ``isotropy condition": $\langle A(w \otimes z), A(w' \otimes z) \rangle =0$ for all $z \in Z$ and all $w,w' \in W$. 
 \end{enumerate}
 \end{proposition}

\begin{proof} The sequence is exact at   
$W \otimes {\mathcal O}_{P^3}(-1)$, 
meaning that the image of the incoming map is equal to the  kernel of the outgoing map, iff the  injectivity condition holds. By duality, the sequence is also exact at  $W' \otimes  {\mathcal O}_{P^3}(1)$ in this case.    
The sequence is a complex at $V \otimes  {\mathcal O}_{P^3}$, meaning that the image of the incoming map is contained in the kernel of the outgoing map,  iff the  isotropy condition holds.    We note that the cohomology of this complex is concentrated at $V \otimes  {\mathcal O}_{P^3}$, the cohomology is zero at $W \otimes {\mathcal O}_{P^3}(-1)$ and at $W' \otimes  {\mathcal O}_{P^3}(1)$.  
\end{proof} 

\noindent  Yang-Mills instantons over $S^4$ correspond self-dual monads that are ``real,"  meaning  that 
the self-dual monads satisfy an additional ``reality condition" \cite{At1,ADHM,BS}.

The paper \cite{BS} studied  Yang-Mills instantons with a certain  $SU(2)$ ``quadrupole symmetry", by constructing  
the corresponding $SU(2)$ equivariant real self-dual monads.    Forgetting the real structure and complexifying 
the $SU(2)$ action on the holomorphic bundles yields the 
 $\sltc$ equivariant instanton bundles used for the construction of Painlev\'e VI solutions discussed in the present paper. 
As previously mentioned   in Section \ref{sec:adhm}, it is known a posteriori \cite{SegU} that 
imposing and subsequently forgetting the reality condition  
produces all   $\sltc$ equivariant instanton bundles for rank $2$, although this is not true for higher rank.    
The  {\it rank} of a self-dual monad is defined to be the rank of the corresponding instanton bundle,  $\rk E = \dim V - 2 \dim W$.  
 
\begin{proposition} 
\label{thm:ssimp}
 (\cite{BS,SegU})  If $\sltc$ acts on a rank $2$ self-dual monad with $Z = \irrep 3$, 
then $W=W(m)$ and $V=V(m)$ for some nonnegative integer $m$, where:   
\begin{align*}  
W(m) = \bigoplus_{\substack{0 \le l \le m-1,\\ l \equiv m-1 \Mod 2}} \irrep {2l}, \qquad\qquad
 V(m) \oplus \irrep {2m-1}& =\bigoplus_{m'-1\le j \le m} \irrep {2 j +1}.
\end{align*}
\end{proposition} 
\noindent It is convenient to consider $V(m)$ as a summand of
\begin{align*}
\hat  V(m) : = V(m) \oplus \irrep {2m-1}=  \irrep {1} \oplus  \irrep{3} \oplus\cdots  \oplus\irrep {2m-1}  \oplus \irrep {2 m +1}. 
\end{align*}
For example,   $\hat V(3) = \irrep 1 \oplus \irrep 3 \oplus \irrep 5 \oplus \irrep 7$, and  cancellation of  
$\irrep{2m -1} = \irrep 5$ yields  
$V(3) = \irrep 1 \oplus \irrep 3 \oplus \irrep 7$. 
 Note that $\hat V(0) = V(0) = \irrep 1$ because $\irrep {-1} := 0$.  
Both $\hat V(m)$ and $V(m)$ have an $\sltc$ invariant $\comp$-bilinear symplectic form $\cbra \cdot \cdot {}$ defined by eqn.\,\ref{eq:4.3}. 
For $W(m)$ it is convenient to separate the odd and even cases. 
For odd $m= 2k+1 \ge 0$, 
\begin{align*}  
W(2k+1) &=\irrep{0} \oplus \irrep{4} \oplus \cdots \oplus \irrep {4k-4}\oplus \irrep{4k}.
\end{align*}
For example, $W(3) = \irrep 0 \oplus \irrep 4$.  
For even $m = 2k \ge 0$, 
\begin{align*}  
W(2k) &=   \irrep{2} \oplus \irrep{6} \oplus \cdots 
\oplus \irrep {4k-6} \oplus \irrep{4k-2}. 
\end{align*}
Note that $W(0) = 0$ because the index set of the sum is empty.

It is  similarly convenient to consider an equivariant instanton bundle $E$ of rank $2$ 
 as a summand of an equivariant instanton bundle $\hat E$ 
 corresponding to the equivariant self-dual monad
\begin{equation*}
\begin{tikzcd}[row sep=normal, column sep=normal]
0 \arrow[r] &  W(m) \otimes {\mathcal O}_{P^3}(-1)  \arrow[r,"\hat A"] &  \hat V(m) \otimes  {\mathcal O}_{P^3}  \arrow[r,"\hat A^o"]& 
 W'(m) \otimes  {\mathcal O}_{P^3}(1)  \arrow[r] & 0. 
\end{tikzcd}
\end{equation*} 
The rank is  $\rk \hat E = \dim \hat V(m) - 2 \dim W(m) =2 + \dim \irrep {2m-1}$.  
Whenever the image of   $\hat A: W(m) \otimes Z \to \hat V(m)$ is contained in $V(m) \subseteq \hat V(m)$,  this 
monad splits as the direct sum
\begin{equation*}
\begin{tikzcd}[row sep=tiny, column sep=normal]
0 \arrow[r] &  W(m) \otimes {\mathcal O}_{P^3}(-1)  \arrow[r,"A"] &   V(m) \otimes  {\mathcal O}_{P^3}  \arrow[r," A^o"]& 
 W'(m) \otimes  {\mathcal O}_{P^3}(1)  \arrow[r] & 0\\
 & \oplus &   \oplus & \oplus \\
0 \arrow[r] & 0 
 \arrow[r] 
&  \irrep{m-1} \otimes  {\mathcal O}_{P^3}  \arrow[r] 
& 0 
  \arrow[r] 
& 0. 
\end{tikzcd}
\end{equation*} 
The equivariant instanton bundle $\hat E$ then splits as the direct sum 
of the desired rank $2$ equivariant instanton bundle $E$ and the trivial equivariant bundle with fiber $\irrep{2m-1}$.     

Following \cite{BS},  we describe  equivariant self-dual monads in terms of a set of 
 coefficients $a_{l,p}$, which are complex numbers in the general case: 

\begin{proposition} 
\label{thm:cofdef}  
(\cite{BS})  Let $W = W(m)$ and $\hat V = \hat V(m)$ and $Z = \irrep 3$. 
Then an equivariant linear map $\hat A: W \otimes Z \to \hat V$ is the direct sum 
\begin{align*} 
\hat A &= \bigoplus_{\substack{0 \le l \le m-1,\\ l \equiv m-1 \Mod 2}} \hat A_l
\end{align*}
of the component maps  
\begin{align*} 
\hat A_l: \irrep {2l} \otimes \irrep 3 \to \hat V, \qquad  w \otimes z \mapsto 
\bigoplus_{0 \le p \le \min(2l,3)} a_{l,p} \, \cbra w z p. 
\end{align*}
\end{proposition}

\begin{proof}  Using Clebsch-Gordan eqn.\,\ref{eq:4.1}, 
we check that $\irrep {2j+1}$ is a summand of  $\hat V(m)$ iff $\irrep{2j+1}$ is a summand of 
$W(m) \otimes \irrep 3$.    Noting that the multiplicity of each summand of $\hat V(m)$ is at most one, 
the assertion follows from Schur's Lemma.   
\end{proof} 
\noindent
Proposition \ref{thm:addef} then translates to conditions on the coefficients $a_{l,p}$ of $\hat A$.  

\begin{theorem}  
\label{thm:eqad}  
(\cite{BS}) 
 An equivariant linear  $\hat A: W \otimes Z \to \hat V$ defines 
an equivariant self-dual linear monad iff its coefficients $a_{l,p}$ satisfy:  
\begin{enumerate}
 \item  The ``injectivity condition": $a_{l,0} \ne 0$ for all $l$. 
 \item The ``isotropy condition" consists of the two conditions. 
 \begin{enumerate}
 \item The ``diagonal isotropy condition":
\begin{align*}
(2 l - 1)^2 \, a_{l,2}^2 &= (2 l + 1) \, a_{l,0}^2 + 2 l (2 l - 3) \, a_{l,1}^2 \qquad \text{whenever $1 \le l $}, \\
(2 l - 1)^2 \, a_{l,3}^2  &= (2 l + 2) (2 l + 5) \, a_{l,0}^2 -  9 (2 l + 1)\, a_{l,1}^2 \qquad \text{whenever  $2 \le l$}. 
\end{align*} 
 \item The  ``off-diagonal isotropy condition": 
 \begin{align*}
 a_{l,0} \, a_{l+2,2} = a_{l,1} \, a_{l+2,3}  \qquad \text{whenever $1 \le l$}.
 \end{align*}
 \end{enumerate}
 \end{enumerate}
\end{theorem}  

\noindent  We did not include the reality condition, which  translates to  \cite{BS}  
the reality of the coefficients $a_{l,p}$.

The existence result Theorem \ref{thm:bsres} now follows from:  

\begin{proposition} 
\label{thm:yes}
(\cite{BS})  For any nonnegative integer $m$, 
the set of coefficients defined as follows satisfies the conditions of 
Theorem \ref{thm:eqad}: 
\begin{align*}
a_{l,0}  &= (2l-1)\sqrt{ 9(2m+1)^2 - (2l+3)^2 }, \qquad 
a_{l,1}  = (2l-1) \sqrt{ (2m+1)^2 - (2l+3)^2 } ,\\
a_{l,2} &= (2l+3) \sqrt{ (2m+1)^2 - (2l-1)^2 }, \qquad
a_{l,3} = (2l+3) \sqrt{ 9(2m+1)^2 - (2l-1)^2 }. 
\end{align*}
Furthermore, $a_{m-1,1} = 0$.  
 \end{proposition} 

\noindent  For each $m$, this set of coefficients explicitly constructs  an equivariant instanton bundle $\hat E_m$ of rank 
$\rk \hat E_m = 2 + \dim \irrep {2m-1}$.  Since $a_{m-1,1} = 0$, this splits as the direct sum of the desired equivariant instanton bundle 
$E_m$ of rank $2$ and the trivial equivariant bundle of rank $2m$.  This explicit construction was used to compute 
the Painlev\'e solutions listed in Section \ref{sec:adhmcompute}.

The uniqueness result Theorem \ref{thm:nonbs} asserts that the solutions of Theorem \ref{thm:yes} are unique up to equivalence. 
For a fixed value of $m$, the coefficient set $\tilde a_{l,p}$ is equivalent to the coefficient set $a_{l,p}$ iff
there exist constants  $\gamma_l$ and  $\kappa_{2j+1}$  such that
\begin{align*} 
\tilde a_{l,p} = \gamma_l\,   \, a_{l,p} \,  \kappa_{2l+3-2p}, 
\end{align*}
where each $\gamma_l \ne 0$ and each $\kappa_{2j+1}^2 = 1$.  
For the real case,  in  $\gamma_l$ are real because the coefficients $\tilde a_{l,p}$ and $a_{l,p}$ are 
real, and in this case the uniqueness up to equivalence was established in \cite{BS}.  
The analogous result also holds for the general case \cite{SegU}, where the coefficients and the $\gamma_l$ may be complex.

\section{The PVI solutions $\Lambda^\pm_m$} 
\label{sec:painleve}

The celebrated {\it sixth equation of Painlev\'e}, 
which, incidentally, was not found by Painlev\'e,   
is a  nonlinear ordinary differential equation  for a function $\lambda(t)$, and depending the complex parameter vector   
$\theta = (\theta_1,\theta_2,\theta_3,\theta_4) \in \comp^4$:
\begin{align}
{d^2 \lambda \over dt^2} = 
{1 \over 2} \left({1 \over \lambda} + {1 \over \lambda -1} + {1 \over \lambda-t}\right)
&\left( {d\lambda \over dt} \right)^2 
- \left( {1 \over t} + {1 \over t -1} + {1 \over \lambda - t} \right) 
{d\lambda \over dt}  \label{eq:1.1theta} \\
+ { \lambda(\lambda-1)(\lambda-t) \over 2 \, t^2 (t-1)^2} 
&\left((\theta_4-1)^2 - \theta_1^2 {t \over \lambda^2} + \theta_3^2{t-1\over(\lambda-1)^2} + 
(1- \theta_2^2){t (t-1) \over (\lambda-t)^2}\right) . 
\nonumber
\end{align}
This form of Painlev\'e VI will be denoted as $P_{VI}(\theta)$.   
We will say that the pair  $\Lambda = \left[ \lambda(t); \theta \right]$ {\it solves} $P_{VI}$ if 
the function $\lambda(t)$ is a solution of the differential equation $P_{VI}(\theta)$.  

Much of the Painlev\'e literature, especially the older literature, parametrizes 
the four constants as: 
\begin{equation*}
C=(\alpha,\beta,\gamma,\delta) := 
\left( 
\tfrac{1}{2} (\theta_4-1)^2, 
- \tfrac{1}{2} \theta_1^2, 
\tfrac{1}{2} \theta_3^2,
\tfrac{1}{2} (1-\theta_2^2)
\right) \in \comp^4.   
\end{equation*}
The resulting equivalent form of the differential equation will be called   {\it classic  Painlev\'e VI}  and will be denoted as $P'_{VI}(C)$: 
\begin{align}
{d^2 \lambda \over dt^2} = 
{1 \over 2} \left({1 \over \lambda} + {1 \over \lambda -1} + {1 \over \lambda-t}\right)
&\left( {d\lambda \over dt} \right)^2 
- \left( {1 \over t} + {1 \over t -1} + {1 \over \lambda - t} \right) 
{d\lambda \over dt}   \label{eq:1.1c}\\
+ { \lambda(\lambda-1)(\lambda-t) \over t^2 (t-1)^2} 
&\left(\alpha + \beta{t \over \lambda^2} + \gamma{t-1\over(\lambda-1)^2} + 
\delta{t (t-1) \over (\lambda-t)^2}\right). \nonumber
\end{align}
\noindent 
It is often both necessary and frustrating to convert between the 
parameters $\theta \in \comp^4$ and $C \in \comp^4$ when comparing results from the Painlev\'e literature.  
A good general reference on Painlev\'e VI is \cite{IKSY}.

For a discussion of algebraic solutions of $P_{VI}$, we 
refer to \cite{LT,Bo2}, which includes an overview of important prior work such as  \cite{IKSY,Mn,H1,H2,H3,H4,D,MW}. 
  One approach for constructing  
 solutions is to exploit the relationship \cite{F,JM} between $P_{VI}$ 
and ``isomonodromic deformations" of linear ODE's with 
singular points, or equivalently of meromorphic connections on 
vector bundles over the complex projective line $ P^1$.   
Hitchin \cite{H2,H3} developed the method of  
``$\sltc$-equivariant compactifications" to construct a certain class of 
isomonodromic deformations, from which he was able to reconstruct 
some solutions of $P_{VI}$ explicitly.  
   
As  observed by 
G. Bor \cite{B}, isomonodromic deformations can also be constructed 
using nontrivial $\sltc$-equivariant vector bundles over 
equivariant compactifications, thereby generalizing 
Hitchin's method.    Implementing this on 
the  monads of Theorem \ref{thm:bsres} yields the 
following result, which had been reported in the unpublished manuscript \cite{Seg1}. 
Defining the parameter vector $\mu := \left( \tfrac{1}{2}, \tfrac{1}{2},\tfrac{1}{2},\tfrac{1}{2} \right) \in \comp^4$, we have: 
\begin{theorem} (\cite{Seg1})
\label{thm:1.1}
For each nonnegative integer $m $, the equivariant instanton bundle $E_m$ yields a pair of explicitly computable 
algebraic Painlev\'e VI solutions 
 $\Lambda_m^+ = \left[\lambda_m^+(t); (2 m + 1) \mu \right]$ and 
$\Lambda_m^- = \left[\lambda_m^-(t); -(2 m + 1) \mu \right]$. 
Each of the algebraic functions  $\lambda_m^\pm(t)$  is expressed implicitly in terms of   
the  rational function 
\begin{align}
t(w) = {\left( 1+w \right)\, \left(-3+ w \right)^3 \over 
 \left( -1+w  \right)\, \left( 3+w \right)^3  }
\label{eq:1.2}
\end{align}
and a rational function of the form 
\begin{align}
\lambda_m^\pm (w) 
& = \left( {{{{\left( -3+w \right) }^2}}\over 
{\left( -1+w \right) \,\left(3+w \right) }}  \right)
{   (-1 + w^2) \, f_m^\pm (w) + 8 \, g_m^\pm (w) \over 
  (3 + w^2) \, f_m^\pm (w) - 24\, g_m^\pm (w)}, 
\label{eq:1.3}
\end{align}
where  
$f_m^\pm$ and $g_m^\pm$ are  even polynomials of degree at most 
$2 m(m+1)$.  
\end{theorem}
\noindent  Differentiation of these algebraic functions is 
straightforward by implicit differentiation.  The first derivative of  
 $\lambda^\pm_m := \lambda^\pm_m(w)$ with respect to $t:=t(w)$ is: 
\begin{equation}
\label{eq:wder}
 \frac{d \lambda^\pm_m}{dt}  =\frac{  \frac{d}{d w} \lambda^+_m(w)}{ \frac{d}{dw}t(w)}
=\frac{(-1+w)^2 (3+w)^4}{16 \, w^2\, (-3+w)^2 }  \,  \frac{d \, \lambda^+_m(w)}{d w} ,  \\
\end{equation} 
The second derivative is expressed similarly.

\subsection{The Explicit PVI Solutions} 
\label{sec:adhmcompute}

We  exhibit here a pair of explicit PVI solutions $\Lambda^\pm_m =\left[ \lambda^\pm_m; \pm (2 m+ 1) \mu \right]$  for each integer $0\le m \le 4$.  Each of these was computed via the 
generalization of Hitchin's logarithmic connection of section \ref{sec:lifts}, 
from  the equivariant instanton bundles $E_m$ discussed in  Sections \ref{sec:adhm} and  \ref{sec:monads}. 
The detailed steps are given in Section \ref{sec:pi}.

\vskip.3cm

\begin{itemize}

\item{$m = 0$:}   The equivariant instanton bundle $E_0$ is the trivial rank $2$ vector bundle over $P^3$, 
and  Hitchin's original method of equivariant compactifications applies without need for generalization. 
The corresponding Painlev\'e pair is 
$\Lambda_0^+  = \left[ \lambda_0^+(t); \mu \right]$ and $\Lambda_0^-  = \left[ \lambda_0^-(t); -\mu \right]$
since $\pm \left(2m+1\right)  = \pm 1$. 
\begin{align}
f_0^+(w) &= 1, \qquad g_0^+ (w) = 0;
\nonumber
\end{align}
\begin{align}
\lambda_0^+(w) 
&= 
{{{{\left( -3+w \right) }^2}\,\left( -1 + {w^2} \right) }\over 
   {\left( -1+w  \right) \,\left( 3+w  \right) \,
     \left( 3 + {w^2} \right) }}.
\label{eq:1.5}
\end{align}
\begin{align}
f_0^-(w) &= 0, \qquad
g_0^-(w) = 1;
\nonumber
\end{align}
\begin{align}
\lambda_0^-(w) 
& = -
{{{{\left( -3+w  \right) }^2}}\over 
   {3\,\left( -1+w  \right) \,\left( 3+w  \right) }}.
\label{eq:1.6}
\end{align} 
\noindent  $\lambda_0^+(t)$ is equivalent by a change of coordinates to Hitchin's 
 Poncelet polygon solution \cite{H3} of Painlev\'e VI for $k=3$. 

\vskip.3cm

\item{$m = 1$:} The generalization of Hitchin's method  to nontrivial bundles, as outlined in  Section \ref{sec:lifts}, is 
necessary for all $m \ge 1$.    
The equivariant instanton bundle  $E_1$ is the ``null correlation bundle" \cite{At1} on $P^3$, with $c_2 = 1$.   
 The corresponding Painlev\'e pair is 
$\Lambda_1^+  = \left[ \lambda_1^+(t); 3 \mu \right]$ and $\Lambda_1^-  = \left[ \lambda_1^-(t); -3\mu \right]$
since $\pm \left(2m+1\right)  = \pm 3$. 
\begin{align}
f_1^+(w) &= 1  ,\qquad
g_1^+(w) = 0;
\nonumber
\end{align}
\begin{align}
\lambda_1^+(w) 
&= 
{{{{\left( -3+w  \right) }^2}\,\left( -1 + {w^2} \right) }\over 
   {\left( -1+w  \right) \,\left( 3+w \right) \,
     \left( 3 + {w^2} \right) }}.
     \label{eq:1.5c}
\end{align}
\begin{align}
f_1^-(w) &=  4 , \qquad
g_1^-(w) =  3 + w^2; 
\nonumber
\end{align}
\begin{align}
\lambda_1^-(w) &= -
{{ {{\left( -3+w  \right) }^2}\,
       \left( 5 + 3\,{w^2} \right)   }\over 
   {5\,\left( -1+w  \right) \,\left( 3+w  \right) \,
     \left( 3 + {w^2} \right) }}.
\label{eq:1.6c}
\end{align}
Comparing eqns.\, \ref{eq:1.5} and \ref{eq:1.5c}  reveals the  coincidence  $\lambda_1^+(t) = \lambda_0^+(t)$;  
the same functions  solves $P_{VI}(\theta)$ for both $\theta = 3 \mu$ and $\theta = \mu$, 
in fact for a one-parameter family of $\theta$.  
There are no additional coincidences among the solutions $\lambda^\pm_m(t)$ computed in 
this section.  
For example, comparing eqns.\, \ref{eq:1.6} and \ref{eq:1.6c}  shows that 
$\lambda_1^-(t) = \lambda_0^-(t)$, and in fact   
$\lambda^-_1(t)$ is a solution  of $P_{VI}(\theta )$ only for the expected value 
$\theta = -3 \mu$.   
\vskip.3cm

\item{$m = 2$:}  The equivariant instanton bundle $E_2$ 
has second Chern class $c_2(E_2) = 3$. 
The corresponding Painlev\'e pair is 
$\Lambda_2^+  = \left[ \lambda_2^+(t); 5 \mu \right]$ and $\Lambda_2^-  = \left[ \lambda_2^-(t); -5\mu \right]$
since $\pm \left(2m+1\right)  = \pm 5$. 
\begin{align*}
f_2^+(w) &= 
12\,{{\left( 3 + {w^2} \right) }^2} , \qquad
g_2^-(w) =   
{{\left( -1 + {w^2} \right) }^2;
}
\end{align*}
\begin{align*}
\lambda_2^+(w) 
&= 
{{{{\left( -3+w\right) }^2}\,\left( -1 + {w^2} \right) \,
     \left( 5 + {w^2} \right) \,\left( 5 + 3\,{w^2} \right) }\over 
   {3\,\left( -1+w \right) \,\left( 3+ w  \right) \,
     \left( 1 + {w^2} \right) \,
     \left( 25 + 6\,{w^2} + {w^4} \right) }}.
\end{align*} 
\begin{align*}
f_2^-(w) &= 
16\,\left( 7 + {w^2} \right) \,\left( 4 + 3\,{w^2} + {w^4} \right),
 \\
g_2^-(w) &= 
\left( 3 + {w^2} \right) \,
  \left( 77 + 89\,{w^2} + 23\,{w^4} + 3\,{w^6} \right); 
\end{align*}
\begin{align}
\lambda_2^-(w) 
& = -
{{ {{\left( -3 +w  \right) }^2}\,\left( 5 + {w^2} \right) \,
       \left( 35 + 63\,{w^2} + 25\,{w^4} + 5\,{w^6} \right) 
         }\over 
   {7\,\left(-1+ w \right) \,\left( 3+w \right) \,
     \left( 1 + {w^2} \right) \,\left( 3 + {w^2} \right) \,
     \left( 25 + 6\,{w^2} + {w^4} \right) }}.
\nonumber
\end{align} 

\vskip.3cm

\item{$m = 3$:} For the sake of 
brevity, we shall from now on leave to the reader the task of substituting the polynomials 
$f_m^\pm (w)$ and $g_m^\pm (w)$ into eqn.\,\ref{eq:1.3}.   
The equivariant instanton bundle $E_3$ has 
second Chern class $c_2(E_3) = 6$.  
The corresponding Painlev\'e pair is 
$\Lambda_3^+  = \left[ \lambda_3^+(t); 7 \mu \right]$ and $\Lambda_3^-  = \left[ \lambda_3^-(t); -7\mu \right]$
since $\pm \left(2m+1\right)  = \pm 7$.  
\begin{align*}
f_3^+(w)&=   
8\,\left( 3381 + 7536\,{w^2} + 6291\,{w^4} + 2576\,{w^6} + 
    611\,{w^8} + 80\,{w^{10}} + 5\,{w^{12}} \right),\cr
g_3^+(w) &= 
{{\left( -1 + {w^2} \right) }^2}\,\left( 3 + {w^2} \right) \,
  \left( 147 + 111\,{w^2} + 57\,{w^4} + 5\,{w^6} \right);
\end{align*}
\begin{align*}
f_3^-(w) &=  
12\,{{\left( 3 + {w^2} \right) }^2}\,
  \left( 3528 + 7272\,{w^2} + 6453\,{w^4} + 2460\,{w^6} + 
    678\,{w^8} + 84\,{w^{10}} + 5\,{w^{12}} \right) ,
    \\
g_3^-(w) &= 
164052 + 590328\,{w^2} + 831465\,{w^4} + 631260\,{w^6} + 
  294435\,{w^8} + 88938\,{w^{10}} 
  \\
&\qquad + 18207\,{w^{12}} + 
  2520\,{w^{14}} + 225\,{w^{16}} + 10\,{w^{18}}.
\end{align*}
\vskip.3cm

\item{$m = 4$:} The equivariant instanton bundle $E_4$ has second 
Chern class $c_2(E_4) = 10$.  
The corresponding Painlev\'e pair is 
$\Lambda_4^+  = \left[ \lambda_4^+(t); 9 \mu \right]$ and $\Lambda_4^-  = \left[ \lambda_4^-(t); -9\mu \right]$
since $\pm \left(2m+1\right)  = \pm 9$. 
\begin{align*}
f_4^+ &= 
4\,( 14619528 + 69918552\,{w^2} + 140631309\,{w^4} + 
    159541866\,{w^6} + 116463663\,{w^8} 
     \\
    &\qquad  + 58384152\,{w^{10}} + 
    20911122\,{w^{12}} + 5489100\,{w^{14}} + 1072278\,{w^{16}} + 
    154176\,{w^{18}}
  \\
&\qquad + 15729\,{w^{20}} + 1050\,{w^{22}} + 
    35\,{w^{24}} ),
    \\
g_4^+ &=
3\, \left( -1 + {w^2} \right)^2 \,\left( 3 + {w^2} \right) \,
  ( 141372 + 402732\,{w^2} + 558819\,{w^4} + 432297\,{w^6} \cr 
&\qquad + 
    209331\,{w^8} + 71361\,{w^{10}} + 16497\,{w^{12}} + 
    2403\,{w^{14}} + 189\,{w^{16}} + 7\,{w^{18}} );
\end{align*}
\begin{align*}
f_4^- &= 
8\,( 326559519 + 1822652766\,{w^2} + 4648210677\,{w^4} + 
    6998194368\,{w^6} + 7025103459\,{w^8}  \\
&\qquad + 5035679226\,{w^{10}} + 
    2678780673\,{w^{12}} + 1084740444\,{w^{14}} + 
    341288829\,{w^{16}} \\
&\qquad + 84427122\,{w^{18}} + 
    16389951\,{w^{20}} + 2449224\,{w^{22}} + 272257\,{w^{24}} + 
    21430\,{w^{26}}\\
&\qquad + 1099\,{w^{28}} + 28\,{w^{30}} ),
 \\
g_4^- &= 
\left( 3 + {w^2} \right) \,
  ( 334968777 + 2143174869\,{w^2} + 
    5776302213\,{w^4} + 8923510233\,{w^6}  \\
&\qquad + 
    8999893881\,{w^8} + 6350646645\,{w^{10}} + 
    3281293773\,{w^{12}} + 1278719217\,{w^{14}} \\
&\qquad + 
    383574771\,{w^{16}} + 89689431\,{w^{18}} + 
    16510551\,{w^{20}} + 2388627\,{w^{22}} + 
    267339\,{w^{24}}  \\
&\qquad + 22239\,{w^{26}} + 1239\,{w^{28}} + 
    35\,{w^{30}} ). 
\end{align*}

\end{itemize}

This completes the list of $\Lambda^\pm_m$ that were explicitly computed from the equivariant instanton bunldes $E_m$ 
of Section \ref{sec:adhm} and Appendix \ref{sec:monads}. 
It is possible to compute $Q^{\pm m} \Lambda^\pm_0$ for any $m$, but the 
equality $\Lambda^\pm_m= Q^{\pm m} \Lambda^\pm_0$ of Theorem \ref{thm:pair} has 
been established only for $m \le 4$ and remains conjectural for $m \ge 5$.    
{\it If}  this equality does hold for all $m$,  {\it then}  any  $\Lambda^\pm_m$ could be recursively computed, 
and for example $\Lambda^+_5$ would be given by eqn.\,\ref{eq:1.3} with:  

\begin{align*}
f^+_5 &=
6 \left(3+ w^2\right)^2\, (4921440381+37977143490 w^2+127613420649 w^4+250673770776 w^6 
 \\ & \qquad
+327148723176   w^8+304141893048 w^{10}+210622703024 w^{12}+112091223944 w^{14}
 \\ & \qquad+46894395098 w^{16}+15666181052
   w^{18}+4231083002 w^{20}+931314344 w^{22}
    \\ & \qquad+167841056 w^{24}+24669272 w^{26}+2918360 w^{28}+271032
   w^{30}+18849 w^{32}+882 w^{34}
       \\ & \qquad+21 w^{36}),\\
g^+_5 & = 
\left( -1 + {w^2} \right)^2 \, (6947915832+44000040942 w^2+133368411033 w^4
 \\ & \qquad+248155844508 w^6+316015211160
   w^8+294283529028 w^{10}+208710837720 w^{12}
    \\ & \qquad+115644336732 w^{14}+50998415472 w^{16}+18167192624
   w^{18}+5280068058 w^{20}
    \\ & \qquad+1254027252 w^{22}+241371320 w^{24}+36993180 w^{26}+4399008 w^{28}+391700
   w^{30}
    \\ & \qquad+24840 w^{32}+1026 w^{34}+21 w^{36} ). 
   \end{align*}

\subsection{The Explicit Okamoto Transformations}
\label{sec:okamoto}

The  diagram of eqn.\,\ref{eqn:zig} summarizes the constructions of this section: 
\begin{equation}
\label{eq:zig}
\begin{tikzcd}[row sep=normal, column sep=large]
 \Lambda_0^+ 
 \arrow[loop left,"B",leftrightarrow] 
 \arrow[dd,leftrightarrow, "R_5" ' ]  &  
 \Lambda_1^+ 
\arrow[dd,leftrightarrow, "R_5" ' ] 
 \arrow[ddl,leftrightarrow,"B" ']&
 \Lambda_2^+ 
  \arrow[dd,leftrightarrow, "R_5" ' ] \arrow[ddl,leftrightarrow,"B" ']& 
 \Lambda_3^+ 
 \arrow[dd,leftrightarrow, "R_5" ' ] \arrow[ddl,leftrightarrow,"B" ']&
 \Lambda_4^+  
 \arrow[dd,leftrightarrow, "R_5" ' ] 
 \arrow[ddl,leftrightarrow,"B" ']\\
&  
&  
& 
& 
\\
 \Lambda_0^- 
 &   \Lambda_1^- 
 & \Lambda_2^- 
 & 
\Lambda_3^- 
& \Lambda_4^- 
& 
     \end{tikzcd}
\end{equation}
\noindent Each double-headed arrow $\leftrightarrow$ denotes an involutive Okamoto transformation between the 
 Painlev\'e solutions $\Lambda^\pm_m$, meaning  that  each of the self-compositions
 $B^2 = B \, B$ and $R_5^2 = R_5 \, R_5$  acts as the identity transformation on 
 $ \Lambda^\pm_m$.   
The Okamoto transformation of Theorem \ref{thm:pair} are then given by the compositions: 
\begin{equation}
 \label{eq:qdef}
 Q := B \, R_5, \qquad Q^{-1} := R_5 \, B.  
 \end{equation}
 
We start with a general discussion of Okamoto transformations,  largely following Boalch \cite{Bo1, Bo3}. 
For the description of the parameters $\theta \in \comp^4$, we had already defined the 
parameter vector $\mu =  ( \tfrac{1}{2},\tfrac{1}{2},\tfrac{1}{2},\tfrac{1}{2}) $ and we further define the standard vectors:  
\begin{equation*}
\varepsilon_1 = (1,0,0,0), \quad \varepsilon_2 = (0,1,0,0), \quad \varepsilon_3 = (0,0,1,0), \quad \varepsilon_4= (0,0,0,1). 
\end{equation*}
Note that $\mu \cdot \mu = 1$  where the  ordinary dot product on $\real^4$  
extends to a $\comp$-bilinear form on $\comp^4$, and that the complex reflection 
 $\theta \mapsto \theta - 2 (\mu \cdot \theta) \, \mu$ appearing 
below is  the complexification of an ordinary reflection through the hyperplane $\mu \cdot \theta = 0$ in $\real^4$.  
We can now introduce the fundamental  Okamoto transformation $R_5$:     

\begin{proposition} (\cite{Oka}) 
\label{thm:boa}
Suppose $\Lambda = \left[ \lambda(t); \theta \right]$ solves  $P_{VI}$.   
  Then $R_5 \Lambda$ as defined below also solves $P_{VI}$:    
\begin{equation*}
R_5\Lambda := \left[ \lambda(t) + { 2 \, (\mu \cdot \theta) \over {(t-1)\lambda'(t) - (\varepsilon_1\cdot \theta) \over  \lambda(t)} + 
{\lambda'(t)- 1- (\varepsilon_2 \cdot \theta) \over   \lambda(t) - t }-
{t \lambda'(t) +( \varepsilon_3 \cdot \theta ) \over  \lambda(t) - 1 } }; \theta - 2 (\mu \cdot \theta) \mu \right].
\end{equation*}
 \end{proposition}  
\noindent 
Here $\lambda'(t)$ denotes the derivative of $\lambda(t)$ with respect to $t$. 
It is an important, but not immediately obvious, fact that $R_5 R_5 \Lambda = \Lambda$ when 
the left side is well-defined.  This includes all the $P_{VI}$ solutions $\Lambda = \Lambda^\pm_m$ of this paper (which are not Riccati solutions, e.g. \cite{Wat}).

 The following Proposition is only stated and proved  only for the finite set of cases $0 \le m \le 4$ because the  proof is based on 
 computations with the explicit  solutions exhibited in section \ref{sec:adhmcompute}.  However, it appears  reasonable 
 to conjecture that the Proposition holds more generally for each nonnegative integer $m$:

 \begin{proposition}  
 \label{thm:r5} For each nonnegative integer $m \le 4$: 
\begin{equation*}
 R_5 \Lambda_m^\pm = \Lambda_m^\mp.
\end{equation*}
 \end{proposition}

\begin{proof} Referring to Theorem \ref{thm:1.1}, 
$\Lambda^\pm_m = \left[ \lambda^\pm_m, \pm (2m+1)\mu \right]$ where 
 $\lambda^\pm_m := \lambda^\pm_m(w)$ is a rational function of $w$.    
 The derivative  $(\lambda^\pm_m)'$  with respect to $t$ is then computed implicitly as in eqn. \,\ref{eq:wder}, and 
 is also expressed as a rational function of $w$.   
 From Proposition \ref{thm:boa} we have: 
\begin{align*}
 R_5 & \left[ \lambda^+_m; \left(2m+1\right) \mu \right]   = \\
&  \left[
 \lambda_m^+  + { 2 \,\left( 2m+1\right) \over {(t-1)( \lambda_m^+)'  - \tfrac{1}{2}\left( 2m+1\right) \over  \lambda_m^+ } + 
{( \lambda_m^+)' - 1-\tfrac{1}{2}\left( 2m+1\right) \over   \lambda_m^+  - t }-
{t ( \lambda_m^+)'  +\tfrac{1}{2}\left( 2m+1\right)\over  \lambda_m^+  - 1 } }; -\left(2m+1\right) \mu \right].
\end{align*}

\noindent  To prove $R_5 \Lambda^+_m = \Lambda^-_m$ it remains to check that
 \begin{equation*}
 \lambda_m^- = \lambda_m^+  + { 2 \,\left( 2m+1\right) \over {(t-1)( \lambda_m^+)'  - \tfrac{1}{2}\left( 2m+1\right) \over  \lambda_m^+ } + 
{( \lambda_m^+)' - 1-\tfrac{1}{2}\left( 2m+1\right) \over   \lambda_m^+  - t }-
{t ( \lambda_m^+)'  +\tfrac{1}{2}\left( 2m+1\right)\over  \lambda_m^+  - 1 } }, 
\end{equation*}
which was verified by explicit computation for each $0 \le m \le 4$ using the 
solutions listed in \ref{sec:adhmcompute}.     Now applying 
$R_5$ to each side yields $R_5 R_5 \Lambda^+_m  =R_5 \Lambda^-_m$, 
which establishes $R_5 \Lambda^-_m =\Lambda^+_m $ for each $0 \le m \le 4$ since $R_5$ is involutive. 
\end{proof}

The group of Okamoto transformations is generated by $R_5$ together with some 
  obvious symmetries of the Painlev\'e VI equation eqn.\,\ref{eq:1.1theta}.  
Since $\varepsilon_j \cdot \varepsilon_j = 1$ for each $1 \le j \le 4$, each of the complex reflections of $\comp^4$ appearing below 
is the complexification of an ordinary reflection through a hyperplane in $\real^4$.  

\begin{proposition}Suppose $\Lambda = \left[ \lambda(t); \theta \right]$ solves $P_{VI}$. 
  Then each of $R_1 \Lambda$, $R_2 \Lambda$, $R_3 \Lambda$, and $R_4 \Lambda$ as defined below 
  also solves $P_{VI}$: 
  \begin{enumerate}
  \item
  $R_1 \Lambda := \left[ \lambda(t);   \theta - 2 (\varepsilon_1\cdot \theta)\,  \varepsilon_1 \right]$. 
   \item
  $R_2\Lambda:=\left[ \lambda(t);\theta - 2 (\varepsilon_2\cdot \theta)\,  \varepsilon_2 \right]$. 
   \item
  $R_3\Lambda := \left[ \lambda(t);\theta - 2 (\varepsilon_3\cdot \theta)\,  \varepsilon_3\right]$.
   \item
  $R_4\Lambda := \left[ \lambda(t);\theta - 2 (\varepsilon_4\cdot \theta-1)\,  \varepsilon_4\right]$.
  \end{enumerate}
\end{proposition}

\begin{proof} In each case the  function $\lambda(t)$  is unchanged, while 
the parameter $\theta \in \comp^4$ is changed.  But in each case  the change of parameter $\theta$ 
does not change the differential equation eqn.\,\ref{eq:1.1theta}:
\begin{enumerate}
\item $P_{VI}(\theta - 2 (\varepsilon_1\cdot \theta)\,  \varepsilon_1) = P_{VI}(\theta)$ because  
$\theta - 2 (\varepsilon_1\cdot \theta)\,  \varepsilon_1 = (-\theta_1,\theta_2,\theta_3,\theta_4)$ 
 and $(-\theta_1)^2=\theta_1^2$. 
\item 
$P_{VI}(\theta - 2 (\varepsilon_2\cdot \theta)\,  \varepsilon_2) = P_{VI}(\theta)$ because  
$\theta - 2 (\varepsilon_2\cdot \theta)\,  \varepsilon_2 = (\theta_1,-\theta_2,\theta_3,\theta_4)$ 
and $(-\theta_2)^2=\theta_2^2$. 
\item 
$P_{VI}(\theta - 2 (\varepsilon_3\cdot \theta)\,  \varepsilon_3) = P_{VI}(\theta)$ because  
$\theta - 2 (\varepsilon_3\cdot \theta)\,  \varepsilon_3 = (\theta_1,\theta_2,-\theta_3,\theta_4)$ 
and $(-\theta_3)^2=\theta_3^2 $. 
\item $P_{VI}(\theta - 2 (\varepsilon_4\cdot \theta-1)\,  \varepsilon_4) = P_{VI}(\theta)$ because  
$\theta - 2 (\varepsilon_4\cdot \theta-1)\,  \varepsilon_4 = (\theta_1,\theta_2,\theta_3,2-\theta_4)$ 
and $((2-\theta_4) -1)^2=(\theta_4 -1)^2  $. 
\end{enumerate}
\end{proof} 
\noindent  Okamoto observed that the group of affine automorphisms of $\comp^4$ generated by 
 the complex reflections of $R_1, \dots, R_5$  is the 
 affine Weyl group of type $D_4$.  Note the complex reflections $R_1$,$R_2$ and $R_3$ preserve the origin, 
 but   $R_4$ does not   
since it reflects  the affine plane  $\varepsilon_4\cdot \theta-1 =0$ does not pass through the origin.

We can now specify the  Okamoto transformation $B$ in the diagram eqn.\,\ref{eq:zig}:  
\begin{align*}
B &:= (R_1 R_2 R_3   R_5) \, R_4  \, (R_5 R_3  R_2  R_1). 
\end{align*}
\noindent  It is easily seen that $B$ is involutive, since 
each $R_j$ is involutive.   
The following Proposition is only stated and proved for the finite set of case  $0 \le m \le 3$ because the  proof is based on 
 computations with the explicit  solutions exhibited in section \ref{sec:adhmcompute}.   However, it is reasonable 
 to conjecture that the Proposition holds more generally for each nonnegative integer $m$:

\begin{proposition}
\label{thm:aprop}
For each nonnegative integer $m \le 3$: 
\begin{equation}
\label{eq:beq}
 B \Lambda_m^- = \Lambda_{m+1}^+. 
 \end{equation}
\end{proposition}    

\begin{proof}  Applying  $(R_5 R_3 R_2 R_1)$ to both sides of eqn.\,\ref{eq:beq} yields the equivalent 
\begin{equation}
\label{eq:expo} R_4 (R_5 R_3 R_2 R_1) \Lambda_m^- = (R_5 R_3 R_2 R_1) \Lambda_{m+1}^+.
\end{equation}
The left side of eqn.\,\ref{eq:expo} expands to  
\begin{align*}  
&R_4 (R_5 R_3 R_2 R_1) \Lambda_m^- =R_4 (R_5 R_3 R_2 R_1) \left[\lambda^-_m; - \left(2m+1\right) \mu \right]
\\
& \quad =\left[
\lambda_m^-  + { \left( 2 m+1 \right) \over {(t-1)( \lambda_m^-)'  - \tfrac{1}{2}\left( 2 m+1\right) \over  \lambda_m^- } + 
{( \lambda_m^-)' - 1 -\tfrac{1}{2}\left( 2 m+1\right) \over   \lambda_m^-  - t }-
{t ( \lambda_m^-)'  +\tfrac{1}{2}\left( 2 m+1\right) \over  \lambda_m^-  - 1 } }, 
(2m + 3) \varepsilon_4 \right]
\end{align*}
\noindent and the right side expands to
\begin{align*} 
& (R_5 R_3 R_2 R_1) \Lambda_{m+1}^+=(R_5 R_3 R_2 R_1) \left[\lambda^+_{m+1};  \left(2m + 3\right) \mu \right]
\\
& \quad =\left[
 \lambda_{m+1}^+  + { -\left(2m + 3\right) \over {(t-1)( \lambda_{m+1}^+)'  + \tfrac{1}{2}\left(2m + 3\right) \over  \lambda_{m+1}^+ } + 
{( \lambda_{m+1}^+)' - 1 + \tfrac{1}{2}\left(2m + 3\right)  \over  \lambda_{m+1}^+  - t }-
{t ( \lambda_{m+1}^+)'  - \tfrac{1}{2}\left(2m + 3\right)  \over \lambda_{m+1}^+  - 1 } } 
;(2m + 3) \varepsilon_4 \right]. 
\end{align*}
\noindent  To prove $B \Lambda_m^- = \Lambda_{m+1}^+$ it remains to check that 
\begin{align*} 
&\lambda_m^-  + { \left( 2 m+1 \right) \over {(t-1)( \lambda_m^-)'  - \tfrac{1}{2}\left( 2 m+1\right) \over  \lambda_m^- } + 
{( \lambda_m^-)' - 1 -\tfrac{1}{2}\left( 2 m+1\right) \over   \lambda_m^-  - t }-
{t ( \lambda_m^-)'  +\tfrac{1}{2}\left( 2 m+1\right) \over  \lambda_m^-  - 1 } }  \\
& \quad =
 \lambda_{m+1}^+  + { -\left(2m + 3\right) \over {(t-1)( \lambda_{m+1}^+)'  + \tfrac{1}{2}\left(2m + 3\right) \over  \lambda_{m+1}^+ } + 
{( \lambda_{m+1}^+)' - 1 + \tfrac{1}{2}\left(2m + 3\right)  \over  \lambda_{m+1}^+  - t }-
{t ( \lambda_{m+1}^+)'  - \tfrac{1}{2}\left(2m + 3\right)  \over \lambda_{m+1}^+  - 1 } },
\end{align*}
which was verified by explicit computation for each $0 \le m \le 3$ using the 
solutions listed in Section \ref{sec:adhmcompute}. 
\end{proof}

The proof of Theorem \ref{thm:pair} now follows from Propositions \ref{thm:r5} and \ref{thm:aprop} applied 
 to the Okamoto transformations $Q := B  \,R_5$ and $Q^{-1} :=  R_5 \, B$ introduced in eqn. \,\ref{eq:qdef}: 
\begin{equation*}
\begin{tikzcd}[row sep=normal, column sep=large]
 \Lambda_0^+   
 \arrow[loop left,"A",leftrightarrow] 1
 \arrow[dd,leftrightarrow, "R_5" ' ] 
 \arrow[r,"Q"]
  &  
 \Lambda_1^+ 
\arrow[dd,leftrightarrow, "R_5" ' ] 
 \arrow[ddl,leftrightarrow,"B" ']
  \arrow[r,"Q"]
  &
 \Lambda_2^+ 
  \arrow[dd,leftrightarrow, "R_5" ' ] \arrow[ddl,leftrightarrow,"B" ']
   \arrow[r,"Q"]
   & 
 \Lambda_3^+ 
 \arrow[dd,leftrightarrow, "R_5" ' ] \arrow[ddl,leftrightarrow,"B" ']
  \arrow[r,"Q"]
  &
 \Lambda_4^+  
 \arrow[dd,leftrightarrow, "R_5" ' ] 
 \arrow[ddl,leftrightarrow,"B" ']\\
&  
&  
& 
&
\\
 \Lambda_0^- 
  \arrow[r,"Q^{-1}" ']
 &   \Lambda_1^- 
  \arrow[r,"Q^{-1}" ']
 & \Lambda_2^- 
  \arrow[r,"Q^{-1}" ']
 & 
\Lambda_3^- 
 \arrow[r,"Q^{-1}" ']
& \Lambda_4^- 
& 
     \end{tikzcd}
\end{equation*}
\noindent 
Note that $Q Q^{-1} \Lambda = \Lambda = Q^{-1} Q \Lambda$ as suggested by the notation, since $R_5$ and 
 $B$ are both involutive.

\section{Isomonodromic Deformations via Hitchin's Logarithmic Connection}
\label{sec:biglog} 

\subsection{On Trivial Equivariant Vector Bundles}

In this section, we generalize Hitchin's method of 
constructing isomonodromic deformations and 
associated solutions of $P_{VI}$.  
Hitchin constructs a natural meromorphic connection 
on a trivial vector bundle over an 
``equivariant compactification" $Z$ of 
a quotient $\sltc/\Gamma$ by a finite subgroup $\Gamma$. 
The monodromy of the 
pullback connection on a rational curve in $Z$ 
is then invariant under deformations of the rational curve \cite{H2, H1}.  
Hitchin  applies a result of Jimbo-Miwa \cite{JM} to extract 
$P_{VI}$ solutions corresponding to such isomonodromic deformation.  
Following an observation of Bor \cite{B}, we describe the 
generaliztion of Hitchin's 
construction to nontrivial vector bundles over $Z$ which admit 
lifts of the $\sltc$ action.  
We refer to Malgrange \cite{Ml} for the general theory of 
isomonodromic deformations.  

We first review
Hitchin's construction of a flat logarithmic, meaning meromorphic with logarithmic singularities, connection on an 
equivariant compacification.   
Let $Z$ be a compact three-dimensional  complex manifold on which $\sltc$ acts 
holomorphically, 
let $Y$ be a (possibly singular) $\sltc$-invariant hypersurface in $Z$, and suppose 
that $Z-Y$ is an $\sltc$ orbit with stabilizer conjugate to 
a finite subgroup $\Gamma \subset \sltc$.    Then $Z$ is 
called an {\it equivariant compactification} of $\sltc/\Gamma$.  
The linearization of the 
$\sltc$ action on $Z$  defines a holomorphic vector-bundle map 
\begin{align}
\alpha:   Z \times\liesltc   \to TZ, \label{eq:2.1} 
\end{align}
where $\liesltc$ is the Lie algebra of $\sltc$, and $TZ$ is the 
tangent bundle of $Z$.   
For a point $q \in Z$, the restriction to fibers over $q$ 
is a linear map 
$\alpha_q: \liesltc \to T_q Z$, which is 
invertible if and only if $q \in Z - Y$.   It follows that 
inverse vector bundle map 
\begin{align*}
 \alpha^{-1}: TZ \to  Z  \times \liesltc, 
 \end{align*}
cannot be holomorphic since it is singular on $Y$, 
in fact $\alpha^{-1}$ is logarithmic.  
Note that $\alpha^{-1}$
may be thought of as   
a meromorphic $\liesltc$-valued one-form on $Z$.    
Let $V$ be an $\sltc$-representation.  
The product $\sltc$ action on $E := Z \times V$ is 
a {\it lift} of the $\sltc$ action on $Z$ to the trivial 
vector bundle $\pi: E \to Z$, meaning that 
$\sltc$ acts on $E$ by vector-bundle morphisms such that the 
induced action on the base space $Z$ coincides with the original  action on $Z$. 
Identifying a section $f$ of the trivial bundle  $E= Z \times V$ 
with the corresponding map $Z \to V$, 
we can define Hitchin's meromorphic connection by  
\begin{align} \nabla f:= d \, f- \sqb{ \alpha^{-1}} {\, f \,} ,  
\label{eq:2.3}
\end{align}
where $\sqb \cdot \cdot: \liesltc \times V \to V$ denotes  
the linearization of the representation $\sltc \times V \to V$.

\subsection{On General Equivariant Vector Bundles}
\label{sec:lifts}

We now generalize the construction of $\nabla$ to  
lifts of an $\sltc$ action on $Z$ to any vector bundle $\pi:E \to Z$,  
without assuming that  $E$ is a trivial bundle $Z \times V$.  
In general, a {\it lift} of the $\sltc$ action to $E$ is an  
$\sltc$ action on $E$ by vector-bundle morphisms such that 
the induced action on the base space $Z$ coincides with the 
original action on $Z$.  
Let $X \in T_q Z$ be a tangent vector at $q \in Z - Y$, and  
let $f$ be a section of $E$.   Define $\nabla$ on $E$ by: 
\begin{align}
 \left(\nabla_X f\right) (q):= \lim_{t \to 0} 
{f( \exp(t \, \alpha_q^{-1}(X) ) \cdot q) 
- \exp(t \, \alpha_q^{-1}(X) ) \cdot f(q) 
\over t},
\label{eq:2.4}
\end{align}
where  the one-parameter subgroup 
$\exp(t \, \alpha_q^{-1}(X) ) \subset \sltc$ acts on  $q \in Z$ in 
the first term, and on 
$f(q) \in E$ in the second term.  Since $\sltc$ acts 
holomorphically on $E$ and $Z$, and since $\alpha^{-1}$ is meromorphic on $Z$, 
$\nabla$ defines a meromorphic connection on $Z$.   
The restriction of $\nabla$ to $Z - Y=\sltc/\Gamma$ is a flat holomorphic connection. 
To verify this, note that $\Gamma$ is finite and that 
the pullback of $\nabla$ by the quotient map  
the quotient map $\sltc \to \sltc/\Gamma$ is a flat holomorphic 
connection  on the trivial bundle $\sltc \times V$.  

It is easy to verify that the connection eqn.\,\ref{eq:2.4} coincides with 
Hitchin's connection eqn.(2.3) when $E$ is trivial.   
On $E = Z \times V$, identify as before the section $f$ of $E$ 
with the corresponding map $Z \to V$, and compute  
\begin{align}
 \left( \nabla_X f\right) (q)&= 
\lim_{t \to 0} { f( \exp(t \, \alpha_q^{-1}(X) ) \cdot q) 
- \exp(t \, \alpha_q^{-1}(X) ) \cdot  f(q) 
\over t} 
\nonumber \\
&= \lim_{t \to 0} { f( \exp(t \, \alpha_q^{-1}(X) ) \cdot q) 
- \left( 
 f(q) + t \sqb{ \alpha_q^{-1}(X)} {  f(q) } + O(t^2)
\right)
\over t} 
\nonumber \\
&=  \left( {\mathcal L}_X  f\right) (q) - 
 \sqb{ \alpha_q^{-1}(X)} {  f(q) } .  \nonumber
 \end{align}
If $V$ is a tangent vector field, 
 $\nabla_V f := i_V (\nabla f)$, and applying the 
contraction $i_V$ to
eqn.\,\ref{eq:2.3} gives   
$$ \nabla_V f = {\mathcal L}_V f- \sqb{ \alpha^{-1}(V)} {\, f \,},   \nonumber $$
establishing the equivalence of the two connections on a trivial bundle.

The restriction of  $\nabla$ 
to $Z-Y$ is flat, so it defines a representation of the  
fundamental group of $Z -Y$.  
More precisely, fixing a basepoint $q \in Z - Y$, 
the holonomy (monodromy) along a path that starts and ends at $q$  depends only on 
the homotopy class of the path, defining the 
{\it monodromy representation} of the flat connection $\nabla$, 
\begin{align*}
h(\nabla,q): \pi_1(Z - Y,q)  \to GL(E_q) , 
\end{align*}
where $GL(E_q)$ is the automorphism group 
of the fiber $E_q := \pi_1^{-1}(q)$.    
A continuous deformation of the basepoint $q$ preserves the 
isomorphism class of $\pi_1(Z -Y, q) \simeq \Gamma$, and also the 
isomorphism class of the monodromy representation $h(\nabla,q)$.

Hitchin's construction of isomonodromic 
deformations from the connection $\nabla$ carries over 
essentially without change to nontrivial bundles $E \to Z$.  
The basic idea is to  
pull back the connection $\nabla$  to a rational curve in $Z$, 
and to consider continuous deformations of the 
rational curve.  A holomorphic map $\kappa : \comp P^1 \hookrightarrow Z$ will 
be called a 
{\it parametrized rational curve}, and the image $\kappa(\comp P^1) \subset Z$ 
will be called the underlying {\it (unparametrized) rational curve}.   
A  curve will 
be called {\it transverse} 
if it is  nonsingular and it intersects the 
hypersurface $Y$ transversely.   
Choose a basepoint $p \in \comp P^1$ such that 
$\kappa (p) \in Z - Y$. 
The  preimage $\kappa ^{-1}(Y) = \{ a_1, a_2, \dots, a_n \}$ is a 
finite subset of $\comp P^1$ which does not contain $p$.  
The pullback $\kappa ^* \nabla$ is a meromorphic connection on 
the  bundle $\kappa ^* E \to \comp P^1$.  
The meromorphic connection is 
{\it logarithmic} iff the pullback to any transverse curve has a simple pole, 
an we will see the simple poles explicitly in our computations.     
 The 
restriction of $\kappa ^* \nabla$ 
to $\comp P^1 -  \{ a_1, \dots, a_n \}$ is a flat 
holomorphic connection with monodromy representation 
$$ h(\kappa ^* \nabla,p): 
\pi_1(\comp P^1 - \{ a_1, \dots, a_n \},p) \to GL((\kappa ^* E)_p) := 
GL(E_{\kappa (p)}).  \nonumber $$
By functoriality, the monodromy representation 
of $\kappa ^* \nabla$ factors through the monodromy representation of 
$\nabla$,  
\begin{align*}
 h(\kappa ^* \nabla, p) = h(\nabla,\kappa (p)) \circ \kappa _* ,
  \end{align*}
where $\kappa _*$ is the homomorphism of fundamental groups induced by 
$\kappa $.    
A continuous transversality-preserving deformation of a transverse 
parametrized rational curve $\kappa $ results in 
the points $\{ p,  a_1, a_2, \dots a_n \}$  
moving on $\comp P^1$ while remaining distinct.
Such a deformation of $\kappa $  induces a continuous  
deformation of   $\kappa _*$ and of $h(\nabla, \kappa (p))$, 
and consequently of the monodromy representation 
$h(\kappa ^* \nabla, p)$. 
The key observation, \cite{H1}, Prop. 6, 
is that the deformation preserves
the isomorphism classes of $\kappa _*$ 
and $h_{\kappa (p)}$, and therefore the 
isomorphism class of the monodromy representation is preserved 
under the deformation of the connection $\kappa ^* \nabla$.  
This is in essence the 
defining property of an {\it isomonodromic deformation} 
of a meromorphic connection, see \cite{Ml} details.

If the bundle $E \to Z$ is nontrivial, the isomorphism class of 
the pullback bundle $\kappa^* E \to \comp P^1$ need not be preserved 
under deformation of the rational curve $\kappa$.  
The change of isomorphism type of $\kappa ^* E$ under deformation of $\kappa $ 
is associated with the term  
``jumping line" in the terminology of  
holomorphic vector bundles \cite{At1,OSS}, and with the 
terms ``$\tau$-function" 
or ``$\tau$-divisor" 
in the terminology of isomonodromic deformations \cite{Ml,JM}.

\subsection{On the Equivariant Instanton Bundles $E_m$}

In this section, we give an explicit formula for the 
meromorphic one-form $\alpha^{-1}$ constructed from 
the $\sltc$ action on the  three-dimensional 
complex projective space $Z := P(\irrep 3)$, 
where $\irrep 3$ is the irreducible four-dimensional 
representation.    

A degree-$d$ homogeneous polynomial on $\irrep 3$ corresponds 
to a linear map from the symmetric product $S^d(\irrep 3)$ to $\comp$, and   
an $\sltc$-invariant homogeneous polynomial corresponds to an 
equivariant map $S^d(\irrep 3) \to \irrep 0$.  From Schur's lemma 
and the following decompositions into irreducibles, 
\begin{align*}
S^1(\irrep 3) &\simeq \irrep 3, \nonumber \\
S^2(\irrep 3) &\simeq \irrep 2 \oplus \irrep 6,  \nonumber \\
S^3(\irrep 3) &\simeq \irrep 3 \oplus \irrep 5 \oplus \irrep 9,  \nonumber \\
S^4(\irrep 3) & \simeq \irrep 0 \oplus \irrep 4 \oplus \irrep 6 \oplus \irrep 8, 
\end{align*}
we conclude that the vector space of invariant degree-$d$ polynomials 
has dimension $0$ if $1 \le d \le 3$, and dimension $1$ if $d=4$.  
The  degree-$4$ invariant polynomial $p$ defined by 
\begin{align*}
 p(u) := \cbra {u^2} {u^2} { } , \qquad u \in \irrep 3,   
 \end{align*}
is nonzero, as one can check by computing $p( x \, y \, (x+y) ) \ne 0$.  
Any  degree-$4$ invariant polynomial on 
$\irrep 3$ is then a scalar multiple of $p$.  

\begin{lemma} 
\label{thm:5.1} 
\hfil\break
\begin{enumerate}
\item 
By the fundamental theorem of algebra, a 
degree-3 homogeneous polynomial $u \in \irrep 3$ can be 
factored as   
a product $u = a\, b \, c$ of degree-1 homogeneous 
polynomials $a,b,c \in \irrep 1$, and we  define  
$$ q(u) := \left(\cbra a b {} \cbra b c {} \cbra c a {} \right)^2  . $$
Then $q(u) = K_q \, p(u)$ for some nonzero constant $K_q \in \comp$.  
\item 
Choose a basis $\{ e_1, e_2, e_3, e_4 \}$ of $\irrep 3$ and 
a basis $\{ \gn 1, \gn 2, \gn 3 \}$ of $\liesltc$. 
The one-dimensional representation $\Lambda^4(\irrep 3) \simeq \irrep 0$ 
has basis $e_1 \wedge e_2 \wedge e_3 \wedge e_4$,  and we  
define $r(u)$ by means of the equation 
$$ \sqb {\gn 1} u \wedge \sqb {\gn 2} u \wedge \sqb {\gn 3} u \wedge u = 
r(u) \, 
e_1 \wedge e_2 \wedge e_3 \wedge e_4 . $$
Then $r(u) = K_r \, p(u)$ for some nonzero constant $K_r \in \comp$ 
(depending on the basis choices).      
\end{enumerate}
\end{lemma}

\begin{proof}  Since $q(u)$ and $r(u)$ are both  
degree-$4$ invariant polynomials on $\irrep 3$,  they are scalar multiples 
of $p(u)$.  Comparing  $q(x \, y \, (x + y))$   and $p(x \, y \, (x + y))$ 
we conclude that $K_q = -1/48 \ne 0$.  
Choosing the basis 
$\{ e_1 = x^3 ,e_2 = x^2 y,e_3 = x y^2 ,e_4 = y^3 \}$ of $\irrep 3$ and 
the basis $\{ \gn 1 =\gz(x,y), \gn 2 = \gp(x,y), \gn 3 = \gm(x,y) \}$ of 
$\liesltc$, we compute  $r(x \, y \, (x + y)) = - 6 \,  p(x \, y \, (x + y))$.  
So $K_r = - 6 \ne 0$ for this choice of bases.  A change of basis in 
$\irrep 3$ induces multiplication of 
 $e_1 \wedge e_2 \wedge e_3 \wedge e_4$ by a nonzero scalar, 
and a change of basis in 
$\liesltc$ induces multiplication of 
$\sqb {\gn 1} u \wedge \sqb {\gn 2} u \wedge \sqb {\gn 3} u \wedge u$ 
by a nonzero scalar, so $K_r$ remains nonzero under change of bases.    
\end{proof}

The  $\sltc$ orbit structure 
on the projectivization $P(\irrep n)$ of 
an irreducible representation $\irrep n$ 
is well-known, see e.g. \cite{MFK} for details.  
Our interest is in the three-dimensional  
projective space $Z = P(\irrep 3)$, which is the quotient of 
$\irrep 3^* := \irrep 3 - \{ 0 \}$ by the equivalence relation 
$$ u \sim \lambda \, u, \qquad \lambda \in \comp^* := \comp - \{ 0 \}. $$
Denoting the quotient map 
by  $\delta: \irrep 3^* \to P(\irrep 3)$,   
a point $q \in P(\irrep 3)$ is an equivalence class 
$q = \delta(u)$, where $u \in \irrep 3^*$ is unique up  
to scalar multiplication.  
There are exactly three 
$\sltc$ orbits:  

\begin{itemize}
\item 
A one-dimensional orbit
consisting of the points $q = \delta(a^3)$, where 
$a \in \irrep 1$ satisfies $a \ne 0$.  
\item 
 A two-dimensional orbit consisting 
of the points $q = \delta (a^2 \, b)$, where 
$a,b \in \irrep 1$ satisfy $\cbra a b {} \ne 0$ (this is equivalent 
to linear independence).    
\item
A three-dimensional orbit consisting of 
the points $q = \delta (a \, b \, c )$, 
where $a, b, c \in \irrep 1$ satisfy $\cbra a b {} \ne 0$, 
$\cbra b c {} \ne 0$, and $\cbra c a {} \ne 0$.  
\end{itemize}

It follows from Lemma \ref{thm:5.1} part 1 that the degree-four hypersurface 
 $Y \subset Z = P(\irrep 3)$ cut out by the 
homogeneous polynomial $p$ is the union of the one-dimensional 
orbit and the two-dimensional orbit.  
The complement $Z - Y = \sltc/\Gamma$ is an $\sltc$ orbit.    
The stabilizer $\Gamma \subset \sltc$  
is the preimage under 
the quotient map $\sltc \to P\sltc$ of a group 
isomorphic to the symmetric group $S_3$, so $\Gamma$ is a finite 
group of order $12$ 
(isomorphic to the ``binary dihedral group" $\tilde D_3$ \cite{H2}.) 

An explicit formula for the meromorphic one-form $\alpha^{-1}$ 
will follow easily from: 

\begin{lemma} 
\label{thm:5.2} 
Suppose that $u$ and $v$ are vectors 
in $\irrep 3$, and that $p(u) \ne 0$.  Then there exists a unique 
vector $c = \alie + \beta \in \liesltc \oplus \comp$ such that  
\begin{align} \sqb \alie u + \beta \, u = v . \label{eq:5.3}
\end{align}
The vector $c = c(u,v) = \alie(u,v) + \beta(u,v)$ is given by the formulae 
$$ \alie(u,v) = {\hlie(u,v) \over p(u)}, \qquad
\beta(u,v) = {\sigma(u,v) \over p(u)}, 
$$
where 
\begin{align}
\hlie(u,v) &= 
\textstyle{1 \over 5} \cbra {\cbra {u^2} u 3} v 2 - \textstyle{2 \over 35}  
\cbra {\cbra {u^2} u 2} v 3 \in \irrep 2 = \liesltc, \nonumber \\
\sigma(u,v) &=  \cbra {\cbra {u^2} u 3} v {} \in \irrep 0 = \comp. \nonumber
\end{align}

\end{lemma}

\begin{proof}
Since $p(u) \ne 0$,   Lemma \ref{thm:5.1} part 2 implies that 
$ \{ \sqb {\gn 1} u , \sqb {\gn 2} u , \sqb {\gn 3} u , u \}$ is a 
basis of $\irrep 3$, so $v$ can be expressed uniquely as a linear 
combination 
\begin{align}
v &= 
c_1 \, \sqb {\gn 1} u  + c_2  \sqb {\gn 2} u + c_3\,  \sqb {\gn 3} u + c_4 \, u, \nonumber \\
\label{eq:5.4}
\end{align}
with complex coefficients $c_i = c_i(u,v)$.  
Defining  
$\alie := c_1 \, \gn 1 + c_2 \, \gn 2 + c_3 \, \gn 3$ and $\beta := c_4$ proves 
the first assertion.

The functions defined by 
\begin{align}
 \hlie(u,v) &:=  p(u) \, \alie(u,v) = p(u) \left( c_1(u,v) \, \gn 1 + 
c_2(u,v) \, \gn 2 + c_3(u,v) \, \gn 3 \right)  
\in \liesltc = \irrep 2, \nonumber \\
\sigma(u,v) &:= p(u) \, \beta(u,v) =   p(u) \, c_4(u,v) \in \comp = \irrep 0.   
\nonumber
\end{align}
are equivariant  because 
eqn.\,\ref{eq:5.3} is equivariant and $p$ is invariant.   
We now determine the homogeneity properties of these equivariant 
functions.   
From eqn.\,\ref{eq:5.4} and the antisymmetry 
of the wedge product we have 
\begin{align}
v \wedge \sqb {\gn 2} u \wedge \sqb {\gn 3} u \wedge u &= 
c_1(u,v) \, \sqb {\gn 1} u \wedge \sqb {\gn 2} u \wedge \sqb {\gn 3} u \wedge u, \nonumber \\
 \sqb {\gn 1} u \wedge v \wedge\sqb {\gn 3} u \wedge u &=
c_2(u,v) \, \sqb {\gn 1} u \wedge \sqb {\gn 2} u \wedge \sqb {\gn 3} u \wedge u 
\nonumber
\end{align}
and similar expressions involving  $c_3$ and $c_4$ (this is 
just Cramer's rule), and 
using Lemma \ref{thm:5.1} part 2 we obtain 
\begin{align}
v \wedge \sqb {\gn 2} u \wedge \sqb {\gn 3} u \wedge u &= 
 K_r \,  p(u) \,c_1(u,v)\, 
e_1 \wedge e_2 \wedge e_3 \wedge e_4  
, \nonumber \\
 \sqb {\gn 1} u \wedge v \wedge\sqb {\gn 3} u \wedge u &=
K_r  \,  p(u) \, c_2(u,v)\, 
e_1 \wedge e_2 \wedge e_3 \wedge e_4 
\label{eq:5.5}
\end{align}
and similar expressions for $c_3$ and $c_4$.  
The left-hand side of each equation eqn.\,\ref{eq:5.5} is clearly 
bihomogeneous of bidegree $(3,1)$ in the variables $(u,v)$, 
so $p(u) \, c_i(u,v)$  must be bihomogeneous of bidegree $(3,1)$, as 
must the linear combinations $\hlie(u,v)$ and $\sigma(u,v)$.  

To analyze equivariant bihomogeneous maps of bidegree $(3,1)$ we use the 
decomposition into irreducibles 
\begin{align}
S^3(\irrep 3) \otimes \irrep 3 
&\simeq \irrep 0 \oplus 2 \, \irrep 2 \oplus 2 \irrep 4 \oplus 
3\, \irrep 6 \oplus 2 \, \irrep 8 \oplus \irrep {10} \oplus \irrep {12} , 
\nonumber
\end{align}
and Schur's lemma.  
The vector space of 
equivariant bidegree $(3,1)$ maps $\irrep 3 \times \irrep 3 \to \irrep 0$ 
then has dimension one, and one checks that 
$\cbra {\cbra {u^2} u 3} v {}$  is a basis, 
so $\sigma(u,v)$ must be scalar multiple. 
Similarly, the vector space of equivariant 
bidegree $(3,1)$ maps $\irrep 3 \times \irrep 3 \to \irrep 2$ 
has dimension two, and one checks that 
$ \{ \cbra {\cbra {u^2} u 3} v 2,   \cbra {\cbra {u^2} u 2} v 3 \}$ is 
a basis, so $\hlie(u,v)$ must be a linear combination.  We omit the 
calculation of the numerical coefficients.  
\end{proof}

We identify the 
tangent bundle $TP(\irrep 3)$ with the quotient of 
$\irrep 3^* \times \irrep 3$ 
by the equivalence relation 
\begin{align}
 (u,v) \sim (\lambda\, u,\lambda \, v+ a \,u), 
\qquad \lambda \in \comp^*,\quad
 a \in \comp, \label{eq:5.6}
\end{align}
and we denote the quotient map by 
 $\gamma:\irrep 3^* \times \irrep 3 \to TP(\irrep 3)$, 
and an equivalence class by $\gamma(u,v) \in TP(\irrep 3)$.  
The main result of this section is: 

\begin{proposition} 
\label{thm:5.3} 
\begin{align}
\alpha^{-1}: TP(\irrep 3) & \to P(\irrep 3) \times \liesltc \nonumber \\
\gamma\left( u, v \right) &\mapsto \left( \delta(u), \alie(u,v) \right) 
\label{eq:5.7}
\end{align}
\end{proposition}

\begin{proof} The vector bundle morphism eqn.\,\ref{eq:2.1} 
defined by the linearization of the $\sltc$ action on $Z = P(\irrep 3)$ 
takes the form 
\begin{align}
\alpha: P(\irrep 3) \times \liesltc & \to 
TP(\irrep 3) \nonumber \\
\left( \delta(u), \glie \right) &\mapsto 
\gamma\left( u, \sqb \glie u \right) . 
\label{eq:5.8}
\end{align} 
We  need to show that the composition of 
the maps eqn.\,\ref{eq:5.8} and eqn.\,\ref{eq:5.7} is the identity on each fiber over 
$Z-Y$, which 
is equivalent to the statement   
$$ \alpha \circ \alpha^{-1}: 
\gamma\left( u, v \right) \mapsto 
\gamma\left( u, v \right) 
$$
whenever $p(u) \ne 0$.  Substituting from eqn.\,\ref{eq:5.7} and eqn.\,\ref{eq:5.8},   
\begin{align}
\alpha \circ \alpha^{-1} : 
\gamma\left( u, v \right) &\mapsto \gamma\left( u, \sqb{\alie(u,v)} v \right), 
\nonumber
\end{align}
and using eqn.\,\ref{eq:5.3} and eqn.\,\ref{eq:5.6}, we obtain the desired result: 
\begin{align}
 \gamma\left( u, \sqb{\alie(u,v)} v \right) &= 
\gamma\left( u, v - \beta(u,v) \, u \right) \nonumber \\
&= \gamma\left( u, v \right) . \qed \nonumber
\end{align}
\renewcommand{\qedsymbol}{} 
\end{proof}

\section{PVI Solutions via Isomonodromic Deformations}
\label{sec:pi}

\subsection{Jimbo-Miwa}
\label{sec:jm}

The relationship between $P_{VI}$ and a certain class of isomonodromic deformations 
has its origins in the work of R. Fuchs \cite{F} in the early part of 
this century, we refer to \cite{IKSY} for a modern discussion.  
Following Hitchin \cite{H2,H1}, we will use a formula of 
Jimbo and Miwa \cite{JM} to extract Painlev\'e solutions from 
the isomondromic deformations discussed in the previous section. 
  
We start with a lift $E \to Z$ of the $\sltc$ action to a rank-two 
vector bundle.  
Consider a   
parametrized rational curve $\kappa: \comp P^1 \to Z$ which satisfies:
\begin{enumerate}
\item 
The underlying rational curve $\kappa(\comp P^1)$ intersects  
the hypersurface $Y$ transversely in $Z$.  
\item 
The pullback 
bundle $\kappa^* E \to \comp P^1$  admits a holomorphic  
trivialization $\kappa^*E \buildrel \simeq \over \rightarrow \comp P^1 
\times V$.
\item 
The preimage $\kappa^{-1}(Y) = \{ 0,1,t,\infty \}$. 
(Note that $t$ is the cross-ratio of the four points.)  
\end{enumerate}
Then express the pullback connection as 
$\kappa^* \nabla  = d + A \, dz $, where $A$ is a meromorphic 
zero-form taking taking values in the vector space $End_0(V)$ of 
traceless endomorphisms of $V$.  Now if 
$A_u \in End_0(V)$ denotes the residue of the 
meromorphic one-form $A \, dz$ at $u \in \comp P^1$, we have 
\begin{align}
A(z) &=  
 {A_0 \over z} + { A_1 \over z-1} +  {A_t \over z - t}
\nonumber \\
&= { \left( A_0 + A_1 + A_t \right) \, z^2 - 
\left( (1+t)\, A_0 + t \, A_1 + A_t \right) \, z + t\, A_0  
\over
z(z-1)(z-t) },  
\label{eq:3.1}
\end{align}
and since the sum  residues  must vanish, 
the pole at $z = \infty$ has residue  
\begin{align*}
 A_\infty = - \left( A_0 + A_1 + A_t \right) \in End_0(V). 
 \end{align*}
Now suppose that the eigenvectors of $A_\infty$ are $\{ k, -k \}$ with $k$ real and positive, 
and let $\{ r^+, r^- \}$ be nonzero  
eigenvectors $A_\infty \, r^\pm = \pm k  \, r^\pm $.  
It is evident from eqn.\,\ref{eq:3.1} that there is (generically) exactly one value of 
$z \in \comp$ such that 
$r^\pm$ is an eigenvector of 
$A(z)$, which will be denoted as $\lambda^\pm := z$.  

Now consider a (holomorphic) 
family of parametrized rational curves $\kappa_{(w)}$ indexed  
by $w$ in an open subset of $\comp$, such that every curve in the 
family satisfies conditions i), ii), and iii) above, and such that 
the cross-ratio $t(w)$ is nonconstant.  Then expressing the function 
$\lambda^\pm(w)$ implicitly as the (generally multi-valued) function  
$\lambda^\pm(t)$, we have:  

\begin{proposition} 
\label{thm:3.1}
 (Jimbo and Miwa \cite{JM}) 
The function $\lambda(t) = \lambda^\pm(t)$ is a solution of the 
classic Painlev\'e VI  $P'_{VI}(C)$, eqn.\, \ref{eq:1.1c},
with parameters 
\begin{equation*}C= \left(\alpha,\beta,\gamma,\delta \right) = \left( {\textstyle{1 \over 2} \left( \pm 2\,k - 1 \right)^2 },
 2 \, \det A_0,
-2 \, \det A_1,
\textstyle{1 \over 2}{\left(1 + 4 \, \det A_t \right) }   \right) \in \comp^4.  
\end{equation*}
\end{proposition}

\noindent We will only need the following special case, translated in terms of the parameters of $P_{VI}(\theta)$, eqn.\,\ref{eq:1.1theta}: 

\begin{corollary}  If each of the four residues $A_0$, $A_1$, $A_t$ and $A_\infty$ has eigenvalues $\{ \tfrac{s}{4}, -  \tfrac{s}{4} \}$ with 
$s$ real and positive, then: 
\begin{enumerate}
\item
$\lambda^+(t)$ solves $P_{VI}(\theta)$ with $\theta = s \mu =\left( \tfrac{s}{2},\tfrac{s}{2},\tfrac{s}{2},\tfrac{s}{2}\right) \in \comp^4$.  
\item
$\lambda^-(t)$ solves $P_{VI}(\theta)$ with $\theta = -s \mu =\left( -\tfrac{s}{2},-\tfrac{s}{2},-\tfrac{s}{2},-\tfrac{s}{2}\right) \in \comp^4$. 
\end{enumerate}
\end{corollary}
\noindent  With our notation, the conclusion of the Corollary becomes: 
  $\Lambda^\pm:= \left[ \lambda^\pm(t); \pm s \mu \right]$ solves $P_{VI}$.  
The corollary applies to the isomonodromic deformation used to construct the Painlev\'e solutions $\Lambda^\pm_m$ of Theorem \ref{thm:1.1}. 
Here each of the four residues $A_0$, $A_1$, $A_t$ and $A_\infty$ has eigenvalues $ \pm \tfrac{s}{4}= \pm\tfrac{1}{4}( 2m + 1)$.  
We will see this from our computations below when  $m = 0$, and refer to Manasliski \cite{Man1}  for $m > 0$.

\subsection{Deforming the Cross-Ratio}

We now construct a family of parametrized 
rational curves $\kappa_{(w)}: \comp P^1 \to P(\irrep 3)$ which satisfy 
the necessary conditions for the application of Proposition \ref{thm:3.1}. 
For each $w$,   
the preimage $\kappa^{-1}_{(w)}(Y)$ should consist of four points, 
and since the degree of the hypersurface $Y \subset P(\irrep 3)$ 
is four, $\kappa_{(w)}$ should be  a family of projective lines.  
The natural 
framework for studying families of projective lines with 
nonconstant cross-ratio $t$ is 
the Geometric Invariant Theory \cite{MFK} quotient of the Grassmannian, 
but for our purposes a more elementary approach suffices. 
     
A pair of linearly independent vectors 
$u$ and $v$ in $\irrep 3$ defines the parametrized projective line  
\begin{align}
\kappa_{u,v}: \comp P^1 &\to P(\irrep 3) \nonumber \\
             z &\mapsto \delta(u + z \, v),   
             \label{eq:6.1}
\end{align}
where $z \in \comp \cup \{ \infty \}= \comp P^1$, and 
$\delta(u + \infty \, v) := \delta(v)$.        

\begin{lemma}
\label{thm:6.1}   
For all but finitely many $w \in \comp P^1$, 
the parametrized projective line 
$\kappa_{(w)} := \kappa_{\tilde u(w), \tilde v(w)}: \comp P^1 \to P(\irrep 3)$ defined by 
the pair of vectors 
\begin{align}
\tilde u (w)  &:= -{ (w+1)\over (w+3)^3} \,(x+ y)^2 \, ( 8 \, x + (w^2 -1) \, y ), \qquad
\tilde v (w)  :=  x^2 y \nonumber
\end{align}
satisfies the properties (1)-(3) of section 3, with $t = t(w)$ 
given by eqn.\,\ref{eq:1.2}.   
\end{lemma}

\begin{proof}The result follows immediately from the following 
identity:     
\begin{align}
p &\left( -(w+3)^3 \, (\tilde u (w)  + z \, \tilde v (w)) \right) \nonumber \\
&=  
- 192\, (w+1)^2 \,(w+3)^6 \, z\, (z -1) \, 
\left((w-1)\,(w+3)^3 \, z - (w+1)\,(w-3)^3 \right).  \nonumber
\end{align} 
The identity is verified by straightforward computation. 
\end{proof}

\subsection{Residue Computations}

In this section we use elementary complex analysis to study the 
meromorphic one-form $\kappa_{u,v}^* \alpha^{-1}$ on $\comp P^1$.  
This plays 
a central role in the construction of 
the meromorphic connection 
$\kappa_{u,v}^* \nabla$ on $\comp P^1$ which in turn is used 
to construct the  
data for Proposition \ref{thm:3.1}. 

Applying Proposition \ref{thm:5.3} to the parametrized rational curve 
defined in eqn.\,\ref{eq:6.1} immediately yields:     

\begin{corollary} 
\label{thm:7.1} 
The meromorphic $\liesltc$-valued one-form 
$\kappa_{u,v}^* \alpha^{-1}$ on $\comp P^1$ is given by: 
$$ \kappa_{u,v}^* \alpha^{-1} 
=  \alie(u+ z v, v) \, dz = {\hlie(u + z v, v) \over p(u + z v)} \, dz .  $$
\end{corollary} 

\noindent Note that if $p(v) = 0$, then $p(u+ z v)$ is a degree-three 
polynomial in the variable $z$, and compare with eqn.\,\ref{eq:3.1}.  
To analyze the residues, consider instead the 
projective line $\kappa_{u,v}$ with  $p(u) = 0$,  and 
expand in a Laurent series at $z = 0$: 

\begin{lemma} 
\label{thm:7.2} 
Suppose that $u = a^2 b$ with $\cbra a b{} \ne 0$. 
Then $\kappa_{a^2 b,v}^* \alpha^{-1}$ has a simple pole at $z = 0$, and 
the residue can be read off from: 
$$ \alie(a^2 b + z v,v) \, dz = 
\left( \textstyle{-1\over 4}\, \gz(a,b)  \, z^{-1} + O(1) \right) \, dz . 
$$
\end{lemma}

\begin{proof} For brevity, we give the proof only for the 
``generic" case  $\cbra v {a^3}  3 \ne 0$.  
Compute the leading term in the Taylor series of the numerator, and simplify;  
\begin{align}
 \hlie(a^2 b + z\, v, v) &= 
 \hlie(a^2 b, v) + O(z) \cr
&= -8 \, \cbra a b {}^2 \, \cbra v {a^3}  {} \, a \, b + O(z) \nonumber \\
&= 8 \,  \cbra a b {}^3 \, \cbra v {a^3}  {} \, \gz(a,b)  + O(z).  \nonumber
\end{align}
Th zeroth order Taylor coefficient of the denominator 
is  $p(a^2 \, b)$, which vanishes by Lemma \ref{thm:5.1} part 1.  
Computing the derivative 
\begin{align}
 \textstyle{d \over dz} \, p(a^2 b + z \, v) &= 
\textstyle{d \over dz} \, \cbra {(a^2 b + z \, v)^2} {(a^2 b + z \, v)^2} {6} \nonumber \\
&= \cbra { 2\, (a^2 b + z \, v) \, v} {(a^2 b + z \, v)^2} {6} + 
\cbra {(a^2 b + z \, v)^2} {2\, (a^2 b + z \, v)\, v} {6} \nonumber \\
&=  \cbra { 2\, (a^2 b + z \, v) \, v} {(a^2 b + z \, v)^2} {}
+ (-1)^6 \, \cbra { 2\, (a^2 b + z \, v) \, v} {(a^2 b + z \, v)^2} {} \nonumber
\end{align}
and evaluating at $z=0$ gives the first-order Taylor coefficient, so 
\begin{align}
 p(a^2 b + z\, v) &=  4\, \cbra {a^2\, b \, v} {a^4 \, b^2} {} \, z 
+ O(z^2) \nonumber \\
&= -32 \cbra a b {}^3 \, \cbra v {a^3} \, z + O(z^2). \quad \qed \nonumber
\end{align}
\renewcommand{\qedsymbol}{}
\end{proof}

We now analyze the numerator $\hlie(u + zv, v)$ of the expression for 
$\kappa^*_{u,v} \alpha^{-1}$ given in  
Corollary \ref{thm:7.1}, compare with the numerator of eqn.\,\ref{eq:3.1}.  

\begin{lemma} 
\label{thm:7.3}
For fixed $u$, $a$, and $b$, the expression 
$$ 
\klie(z) :=
\hlie(u + z \, a^2 b, a^2 b)  + 
8 \cbra a b {}^3 \, \cbra u {a^3} {} \, \gz(a,b)\, z^2 
$$
is a degree-one ($\liesltc$-valued) 
polynomial function of the variable $z$.  
\end{lemma}

\begin{proof}
The identity $\kappa_{u,v}(z^{-1})= \kappa_{v,u}(z) $ is immediate 
from eqn.\,\ref{eq:6.1}, so Corollary \ref{thm:7.1} implies 
\begin{align}
  \alie(u + z^{-1} v, v) \, dz^{-1}= \alie(v + z u,u) \, dz. \label{eq:7.1}
  \end{align}
Now substituting this into the tautology
\begin{align}
0 &= {\hlie(u + z^{-1} \, a^2 b,a^2 b) \over p(u+ z^{-1} \, a^2 b) } -  
 \, \alie( u+ z^{-1} a^2 b,a^2 b) \nonumber \\
&= z^2 \left(\,  \hlie(u + z^{-1} \, a^2 b,a^2 b)  -  
 p(u+ z^{-1} \, a^2 b) \, \alie( u+ z^{-1} a^2 b,a^2 b)\right) \, d z^{-1} \nonumber
\end{align}
and using the homogeneity of $p$, 
\begin{align}
0 
&=  z^2 \, \hlie(u + z^{-1} \, a^2 b,a^2 b) \,\, d\,z^{-1} 
-  z^{-2} \, p(a^2 b + z \, u) \, \alie(a^2 b + z u, u) \, dz,  
\nonumber
\end{align}
and Lemma \ref{thm:7.2} gives 
\begin{align}
0 
&= z^2 \, \hlie(u + z^{-1} \, a^2 b,a^2 b) \,\, d\,z^{-1} \nonumber \\
&\qquad -  z^{-2} \, 
\left( -32 \, \cbra a b {}^3 \, \cbra u {a^3} {} \, z + O(z^2) \right)
\, \left(\textstyle{-1 \over 4} \, \gz(a,b) \, z^{-1} + O(1) \right) \, 
\left(-z^2 \, d\, z^{-1} \right) \nonumber \\
&= \left( z^2 \, \klie(z^{-1}) 
 + O(z) \right) \, 
d \, z^{-1} . \nonumber
\end{align}
Now $z^2 z^{-n} = O(z)$ only if $2-n \ge 1$, so the polynomial $\klie(z)$ 
must have degree $n \le 1$.  
\end{proof}

\noindent 
For fixed $u$, $a$, and $b$, we write the numerator in terms of 
the $\liesltc$ basis eqn.\,\ref{eq:4.5}:  
\begin{align}
\hlie(u + z a^2 b, a^2 b) = h_0(z) \, \gz(a,b)  +
h_+(z) \, \gp(a,b) +
h_-(z) \, \gm(a,b),  \label{eq:7.2}
\end{align}
where the coefficient functions $h_0(z)$, $h_\pm(z)$ depend on the 
choice of $u$, $a$, and $b$.

\begin{proposition}
\label{thm:7.4} 
\begin{align}
\cbra a b {}^3 \, h_0(z)  
&= 3\, \buadb^2\, \buabd-2\, \buat\, \buabd^2 - \buat\, \buadb\, \bubt
\cr
&\qquad + 
2\, \bab^3\, ( 4\, \buat \, \buabd -3\, \buadb^2)\, z  \nonumber \\
&\qquad -  
8\, \bab^6\, \buat \, z^2, 
\nonumber \\
\cbra a b {}^3 \, h_+(z) 
&= 
4\, \bubt\, 
\left(
 \buadb ^2 -\buat \, \buabd  
    + 2\, \bab^3\, \buat \, z
\right), 
\nonumber \\
\cbra a b {}^3 \, h_-(z) 
&= 
2\,\buat \, 
\left( 
  \buat \, \bubt -\buadb \, \buabd
  + 2\, \bab^3 \, \buadb\, z
\right).
\nonumber 
\end{align}
\end{proposition}

\begin{proof}
Applying eqn.\,\ref{eq:4.4} to Lemma \ref{thm:7.3},  
\begin{align}
h_0(z) + 8 \cbra a b {}^3 \, \cbra u {a^3} {} \,  z^2  &= 
-\cbra {\klie(z)} \gz {}= O(z),{} 
\cr
h_+(z)  &= 
-2 \, \cbra {\klie(z)} \gm  {} = O(z),\nonumber \\
h_-(z)  &= 
- 2\, \cbra {\klie(z)} \gp {} = O(z).  \nonumber
\end{align}
It remains compute and simplify the zeroth and first-order Taylor coefficients, 
as in the proof of Lemma \ref{thm:7.2}.   We omit the details for the sake of brevity. 
\end{proof}

\subsection{The Example of $\Lambda^\pm_0$ in detail} 
\label{sec:triv}

We now combine the results of previous sections to compute the two solutions 
$\Lambda_0^+ = \left[ \lambda_0^+(t), \mu \right]$  and $\Lambda_0^-= \left[ \lambda_0^-(t), -\mu \right]$ of 
$P_{VI}$ arising from  the
 the trivial bundle $E = P(\irrep 3) \times \irrep 1$ with 
the product $\sltc$ action. 
The pullback bundle 
$\kappa^*_{u,v} E = \comp P^1 \times \irrep 1$ is trivial 
for any parametrized projective line $\kappa_{u,v}: \comp P^1 \to P(\irrep 3)$,  
and the pullback connection is described by:    

\begin{lemma}
\label{thm:8.1} 
$$ \kappa_{u,v}^* \nabla = d + A \, dz $$
where the $End_0(\irrep 1)$-valued zero-form $A$ is given by 
$$ A(z) = - \sqb {\alie(u + z v, v)} {\, \cdot \, } . $$
\end{lemma}

\begin{proof}
The pullback of the connection $\nabla$ defined by eqn.\,\ref{eq:2.3}, 
acting on a section $h$ of $P^1 \times \irrep 1$,  is  
$$ (\kappa^*_{u,v} \nabla) h = d\, h - \sqb {\kappa^*_{u,v} \alpha^{-1}} h , $$
and by Corollary \ref{thm:7.1} 
$$ (\kappa^*_{u,v} \nabla) h= d \, h - \sqb {\alie(u+zv,v)} {h \,} \,dz 
 . \qed$$
\renewcommand{\qedsymbol}{}
\end{proof}
\noindent We now apply this result 
to the family of projective 
lines described in Lemma \ref{thm:6.1}, and then use the results of Section 7 to 
calculate the  data that enters into the   
Jimbo-Miwa formula Proposition \ref{thm:3.1} for $P_{VI}$ solutions 
(although not strictly in this order).  
 
We may then suppose that $v = a^2 b$ with $\cbra a b {} \ne 0$.  
From Lemma \ref{thm:7.2}  (see also eqn.\,\ref{eq:7.1}) 
we compute  the residue of $A \, dz$ at $z = \infty$ 
to be the endomorphism 
$$ 
A_\infty = \textstyle{1 \over 4} \sqb {\gz(a,b)} {\, \cdot \,} : 
\irrep 1 \to \irrep 1 , $$
and from eqn.\,\ref{eq:4.8} we see that $A_\infty$ has eigenvalues 
$\{ \tfrac{1}{4}, -\tfrac{1}{4} \}$ with corresponding 
eigenvectors $\{ r^+ = a, r^- = b \}$.    
Next we need to 
solve for  $z = \lambda^\pm \in \comp$ such that 
$r^\pm$ is an eigenvector of $A(\lambda^\pm)$.  

\begin{lemma} 
\label{thm:8.2} 
A vector $r \in \irrep 1$ is an eigenvector of 
the linear map   
$B: \irrep 1 \to \irrep 1$ if and only if  $\cbra r {B r} {} = 0$.  
\end{lemma}

\begin{proof} This holds for any two-dimensional vector space with 
a nondegenerate antisymmetric bilinear form. 
\end{proof} 

\noindent So we need to solve for $z$ the equation 
\begin{align*}
 0 = \cbra {r^\pm} {\sqb{A(z)} {r^\pm}} {} , 
 \end{align*}
which by Lemma \ref{thm:8.1} and the definition of $\alie$  
(Lemma \ref{thm:5.2}) is equivalent to
\begin{align}
0 = \cbra {r^\pm} {\sqb{\hlie(u + z a^2 b, a^2 b)} {r^\pm}} {} . 
\nonumber
 \end{align}
Using the basis expansion eqn.\,\ref{eq:7.2} and the fact that $r^\pm$ is an eigenvector of 
$\sqb{ \gz(a,b)} {\, \cdot \,} $, this is equivalent to
\begin{align} 0 =  h_-(z) \, f_0^\pm - h_+(z) \, g_0^\pm ,  
\label{eq:8.2}
 \end{align}
where  
\begin{align}
f_0^\pm &:=  \cbra {r^\pm} {\sqb {\gm(a,b)} {r^\pm} } {}  , \cr
g_0^\pm &:= -\cbra {r^\pm} {\sqb {\gp(a,b)} {r^\pm} } {} .
\nonumber
\end{align}

Now specialize to the paramerized projective line $\kappa_{(w)}$ of 
Lemma \ref{thm:6.1}.  Substituting  
$u = \tilde u (w)$, $a = x$, $b = y$ into the 
formulae for $h_\pm(z)$ from Proposition \ref{thm:7.4} yields 
\begin{align}
h_+(z) &= -{96 \,(w+1)^3\over (w+3)^7} 
\, 8 \, \left((w-3)^2 + 3\, (w-1)\, (w+3)\, z)\right)  \cr
h_-(z) &= {96\,  (w+1)^3 \over (w+3)^7} (w-1)\,\left((w-3)^2\, (w+1) 
- (w+3)\, (3 + w^2)\, z)\right).  
\nonumber
\end{align}%
Then eqn.\,\ref{eq:8.2} is solved by  
 $z = \lambda_0^\pm(w)$, where (compare with eqn.\,\ref{eq:1.3})
\begin{align}
\lambda_0^\pm(w)   
& := \left( {{{{\left( w-3 \right) }^2}}\over 
{\left( w-1 \right) \,\left(w+3 \right) }}  \right)
{   (-1 + w^2) \, f_0^\pm + 8 \, g_0^\pm \over 
  (3 + w^2) \, f_0^\pm - 24\, g_0^\pm}.  
\label{eq:8.3}
\end{align}
Now $r^+ = a = x$ and $r^- = b = y$, so computing   
\begin{align}
 f_0^+ &= \cbra x {\sqb{\gm(x,y)} x} {} = \cbra x y {} = 1, \cr 
g_0^+ &= -\cbra x {\sqb{\gp(x,y)} x} {} = 0,  
\nonumber
\end{align}
and  substituting into eqn.\,\ref{eq:8.3} we into we obtain eqn.\ref{eq:1.5}, and similarly 
from 
\begin{align}
f_0^- &= \cbra y {\sqb{\gm(x,y)} y} {} = 0,  \cr
g_0^- &= -\cbra y {\sqb{\gp(x,y)} y} {} = -\cbra y x {} = 1.  
\nonumber
\end{align}
we obtain eqn.\,\ref{eq:1.6}.  From Proposition \ref{thm:3.1} we conclude that 
$\Lambda_0^\pm = \left[ \lambda_0^\pm(t) ; \pm \mu \right]$  solves $P_{VI}$. 
Here the parameter $\theta = \pm \mu \in \comp^4$ is computed from the residues.   

\end{document}